%% file: manuscript-numapde-preprint.tex
\pdfoutput=1
\documentclass{numapde-preprint}

\usepackage{numapde-semantic}
\usepackage{numapde-style}
\usepackage{numapde-local}

\hypersetup{
	pdftitle={Optimal Control of Hughes' Model for Pedestrian Flow via Local Attraction},
	pdfauthor={Roland Herzog, Jan-Frederik Pietschmann, Max Winkler},
	pdfkeywords={nonlinear transport, Eikonal equation, ODE-PDE coupling, optimal control, pedestrian motion}
}

\title{Optimal Control of Hughes' Model for Pedestrian Flow via Local Attraction}
\subtitle{}
\shorttitle{Optimal Control of Hughes' Model}

\author{Roland Herzog\thanks{Technische Universität Chemnitz, Faculty of Mathematics, 09107 Chemnitz, Germany (\email{roland.herzog@mathematik.tu-chemnitz.de}, \url{https://www.tu-chemnitz.de/mathematik/part_dgl/people/herzog}, \orcid{0000-0003-2164-6575}, \email{jfpietschmann@mathematik.tu-chemnitz.de}, \url{https://www.tu-chemnitz.de/mathematik/invpde/}, \orcid{0000-0003-0383-8696}, \email{max.winkler@mathematik.tu-chemnitz.de}, \url{https://www.tu-chemnitz.de/mathematik/part_dgl/people/winkler}, \orcid{0000-0002-5292-2280}).}
\and
Jan-Frederik Pietschmann\footnotemark[1]
\and
Max Winkler\footnotemark[1]}
\shortauthor{R. Herzog, J.-F. Pietschmann and M. Winkler}

\dedication{}

\IfFileExists{numapde-uncolorizeHyperrefIfChanges.sty}{\RequirePackage{numapde-uncolorizeHyperrefIfChanges}}{}

\begin{document}
\maketitle

% Insert abstract
\begin{abstract}
We discuss the control of a human crowd whose dynamics is governed by a regularized version of Hughes' model, cf.~Hughes: \eqq{A continuum theory for the flow of pedestrians}. \emph{Transportation research part~B: methodological}, 36 (2002).
We assume that a finite number of agents act on the crowd and try to optimize their paths in a given time interval. 
The objective functional can be general and it can correspond, for instance, to the desire for fast evacuation or to maintain a single group of individuals. 
We provide an existence result for the forward model, study differentiability properties of the control-to-state map, establish the existence of a globally optimal control and formulate optimality conditions.\end{abstract}

% Insert keywords
\begin{keywords}
nonlinear transport, Eikonal equation, ODE-PDE coupling, optimal control, pedestrian motion\end{keywords}

% Insert Mathematics Subject Classification (MSC2010)
\begin{AMS}
\href{https://mathscinet.ams.org/msc/msc2010.html?t=35Q93}{35Q93}, \href{https://mathscinet.ams.org/msc/msc2010.html?t=49J20}{49J20}, \href{https://mathscinet.ams.org/msc/msc2010.html?t=35K20}{35K20}, \href{https://mathscinet.ams.org/msc/msc2010.html?t=92D50}{92D50}
\end{AMS}

% Insert document body
\input{main.tex}

% Insert bibliography
\printbibliography

\end{document}

%% file: main.tex
\section{Introduction}\label{sec:intro}
The starting point of this work is Hughes' model for the movement of a (large) crowd of pedestrians introduced in \cite{Hughes:2002:1}. 
Its unknowns are the density $\rho = \rho(t,x)$ and the potential $\phi = \phi(t,x)$ functions for $x \in \Omega \subset \R^2$ and $t \in (0,T)$.
With the space-time cylinder denoted by $Q_T \coloneqq (0,T) \times \Omega$, the model reads 
\begin{subequations}
	\label{eq:hughes}
	\begin{alignat}{2}
		\label{eq:hughes1}
		\partial_t\rho - \nabla\cdot\paren[auto](){\rho \, f(\rho)^2 \, \nabla \phi} 
		&
		= 
		0 
		&
		\quad 
		&
		\text{in }Q_T,
		\\
		\label{eq:hughes2}
		\abs{\nabla\phi} 
		&
		= 
		\frac{1}{f(\rho)} 
		&
		&
		\text{in } Q_T.
	\end{alignat}
\end{subequations}
In the simplest case, the model is supplemented with homogeneous Dirichlet boundary conditions for $\phi$ and $\rho$. 
Due to the hyperbolic nature of the first equation, the boundary conditions for $\rho$ have to be posed in a suitable sense (\ie, on a generalized inflow part involving the function $f$). 
While many models for pedestrian dynamics are microscopic in the sense that they provide constitutive laws for the motion of each pedestrian (\eg, systems of ODEs or cellular automata or the social force model, cf.\ \cite{BursteddeKlauckSchadschneiderZittartz:2001:1,HelbingMolnar:1995:1}), Hughes' model starts from a macroscopic approach. 
It is based on the following three assumptions:
\begin{enumerate} 
	\item The velocity $v$ of the pedestrians is determined by the
	density $\rho$ of the surrounding pedestrian flow and the behavioral
	characteristics of the pedestrians only. Denoting the movement
	direction by $u\in \R^2$ there holds
		\begin{equation*}
			v = f(\rho) \, u
			,
			\quad 
			\abs{u}
			=
			1
			.
		\end{equation*}
	\item Pedestrians have a common sense of the task (called potential $\phi$), \ie, they aim to reach their common destination by
		\begin{equation*}
			u = -\frac{\nabla\phi}{\abs{\nabla\phi}}.
		\end{equation*}
	\item Pedestrians seek to minimize their (accurately) estimated travel time, but modify their velocity to avoid high densities. The potential is thus a solution of the Eikonal equation
		\begin{equation*}
			\abs{\nabla\phi}=\frac{1}{f(\rho)}.
		\end{equation*}
\end{enumerate}
Combining these rules yields the model \eqref{eq:hughes}. 
For brevity we write
\begin{equation}
	\label{eq:movement_direction}
	\beta(\rho,\phi) 
	\coloneqq
	f(\rho)^2\,\nabla\phi
	.
\end{equation}
Due to the previous explanations it is clear that $\restr{-\beta(\rho,\phi)}{(t,x)}$ is the direction of the individuals in the point $(t,x)\in Q_T$.

The analysis of this system is quite involved since the derivative of $\phi$, being the viscosity solution to \eqref{eq:hughes2}, has jump discontinuities on a set that depends on $\rho$ and is not known a~priori. 
Since we are going to consider optimal control problems, we shall focus on a regularized version of the forward model \eqref{eq:hughes}, given by
\begin{subequations}
	\label{eq:hreg}
	\begin{alignat}{2}
		\label{eq:hreg1}
		\partial_t\rho - \nabla \cdot\paren[big](){\rho\,\beta(\rho,\phi)} 
		&= 
		\eps \, \laplace \rho
		&
		\quad 
		&
		\text{in }Q_T,
		\\ 
		\label{eq:hreg2}
		-\delta_1\,\laplace \phi + \abs{\nabla\phi}^2 
		&=
		\frac{1}{f(\rho)^2+\delta_2}
		&
		\quad 
		&
		\text{in }Q_T
		.
	\end{alignat}
\end{subequations}
Here $\varepsilon$, $\delta_1$ and $\delta_2$ are positive and fixed regularization parameters.
The model is supplemented with initial conditions
\begin{align}\label{eq:hreg_initial}
	\rho(x,0)  
	= 
	\rho_0(x)
	\quad 
	\text{in } \Omega
\end{align}
and mixed boundary conditions
\begin{equation}\label{eq:forward_bc}
	\begin{aligned}
		-\paren[big](){\eps\,\nabla\rho  + \rho\,\beta(\rho,\phi)} \cdot n 
		&
		=
		\eta\,\rho 
		,
		& 
		\phi 
		&
		= 
		0 
		&
		& 
		\text{on } \Sigma_\textup{D}
		,
		\\
		\paren[big](){\eps\,\nabla\rho  + \rho\,\beta(\rho,\phi)} \cdot n 
		&
		=
		0 
		,
		& 
		\nabla \phi \cdot n 
		&
		=
		0 
		&
		&
		\text{on } \Sigma_\textup{W}
		,
	\end{aligned}
\end{equation}
where we assume that the boundary consists of parts which act as \eqq{doors} $\partial\Omega_\textup{D}$, at which pedestrians can exit with a given outflow velocity $\eta>0$, and \eqq{walls} $\partial\Omega_\textup{W}$.
We assume $\partial\Omega_\textup{D} \cup \partial\Omega_\textup{W} = \partial\Omega$ and $\partial\Omega_\textup{D} \cap \partial\Omega_\textup{W} = \emptyset$ and set $\Sigma_\textup{W} \coloneqq (0,T) \times \partial\Omega_\textup{W}$ and $\Sigma_\textup{D} \coloneqq(0,T) \times \partial\Omega_\textup{D}$.

\subsection{Optimal Control Problem}
\label{subsec:optimal_control_problem}
Our optimal control problem is based on the following scenario.
We assume that there is a small, given number $M>0$ of agents (guides), who are able to locally influence the motion of pedestrians in their vicinity. 
Think, for instance, of tourist guides or marked security personnel at large sports events. 
The $i$-th agent's position is described by a function $x_i(t) \in \R^2$,
$i=1,\ldots, M$, and the interaction is modeled by a radially symmetric and
decreasing kernel $K(x) = k(\abs{x})$, which enters \eqref{eq:hreg1} as an
additional potential given by
\begin{equation}
	\label{eq:potential_term}
	\phi_K(\bx;x) 
	= 
	\sum_{i=1}^M K\paren[big](){x - x_i(t)}
	=
	\sum_{i=1}^M k\paren[big](){\abs{x - x_i(t)}}
	.
\end{equation}
Here and throughout, $\bx =  \bx(t) = (x_1(t), \ldots, x_M(t))^\transp$ is the function collecting all agent positions.
Furthermore, we need to modify the model to insure that the maximal velocity of the crowd is still normalized to the velocity model $f(\rho)$ despite the agents' presence. 
Indeed, for the unregularized model, 
\begin{equation*}
	f(\rho)^2\abs{\nabla\phi } 
	= 
	f(\rho) 
\end{equation*}
holds.
When $\phi$ is replaced by $\phi + \phi_K$, this is no longer true. 
Thus we explicitly normalize the transport direction, \ie, instead of \eqref{eq:movement_direction}, we define the transport velocity by
\begin{equation}
	\beta(\rho,\phi,\bx)
	\coloneqq
	f(\rho) \, h\paren[big](){\nabla (\phi + \phi_K(\bx;\cdot))}
	.
\end{equation}
Here, $h$ is a smoothed projection onto the unit ball. 
Throughout this article we will use
\begin{equation*}
	h(\by) 
	= 
	\minepsilon \paren[auto]\{\}{1,\abs{\by}} \frac{\by}{\abs{\by}}
\end{equation*}
with $\minepsilon$ a fixed smooth approximation of the minimum function.

We assume that the agents move with maximal possible velocity towards a prescribed directions
$u_i(t) \in \R^2$, $i=1,\ldots,M$, which act as control in the system. 
Since the agents are also part of the crowd, their effective velocity will depend on the surrounding density in the same way as it does for all other individuals in the crowd. 
We therefore assume $\abs{u_i(t)} \le 1$ and the law of motion for the agents will be
\begin{align}\label{eq:ODE}
	\dot x_i(t) 
	= 
	f\paren[big](){\overline\rho(t,x_i(t))}\,u_i(t),\quad i=1,\ldots, M,
\end{align}
where $\overline\rho$ is a suitable extension of $\rho$ from $\Omega$ to $\R^2$ which will be detailed later.
While this extension is necessary to ensure the existence of solutions since the agents may leave the domain~$\Omega$ on which $\rho$ is defined, the precise choice of the extension is clearly irrelevant in terms of modeling.
Notice also that the ODE \eqref{eq:ODE} does not prevent agents from walking through walls or to attract people from behind a wall. The first issue can be avoided by imposing additional state constraints in an optimal control problem.

The complete forward system which we are going to consider finally reads 
\begin{subequations}
	\label{eq:forward_system}
	\begin{alignat}{2}
		\label{eq:forward_system1}
		\partial_t\rho - \nabla\cdot\paren[big](){\rho\,\beta(\rho,\phi,\bx)} 
		&
		= 
		\eps \, \laplace \rho
		&
		&
		\quad
		\text{in } Q_T,
		\\
		\label{eq:forward_system2}
		-\delta_1\,\laplace \phi + \abs{\nabla\phi}^2 
		&
		= 
		\frac{1}{f(\rho)^2+\delta_2}
		&
		&
		\quad
		\text{in } Q_T, 
		\\
		\label{eq:forward_system3}
		\dot x_i(t) 
		&
		=  
		f\paren[big](){\overline\rho(t,x_i(t))} \, u_i(t)
		&
		&
		\quad
		\text{for } t\in(0,T), \quad i=1,\ldots,M,
	\end{alignat}
\end{subequations}
together with initial condition \eqref{eq:hreg_initial} on $\rho$ and 
\begin{align}\label{eq:ODE_init}
	x_i(0) =  x_{i,0}
	\quad 
	\text{ for } i = 1,\ldots, M,
\end{align}
as well as boundary conditions \eqref{eq:forward_bc}.

The aim of the present paper is to investigate several optimal control problems for this coupled system. We seek an optimal control function $\bu$ such that the solution triple $\by = (\rho,\phi,\bx)$ is optimal in a certain sense. Depending on the application in mind it remains to define a suitable objective functional. We particularize the following two examples:

\begin{itemize}
\item \textit{Minimal evacuation time}:
In this case one seeks to minimize the time required for the evacuation of a 
room. 
The exits where individuals can leave the room are located at the boundary 
parts $\partial\Omega_\textup{D}$ modeling doors.
As time-optimal control problems with PDEs are rather challenging, see, \eg, 
\cite{BonifaciusPieperVexler:2019:2,Knowles:1982:1,Schittkowski:1979:1,ZhengYin:2014:1},
 we consider a simpler but closely related model. 
We fix a reasonably large final time $T>0$ and minimize 
\begin{equation}
\label{eq:objective_optimal_evacuation}
J(\rho,\phi,\bx;\bu)
\coloneqq
c_1 \int_\Omega \rho(T,x) \dx + c_2 \int_0^T\int_\Omega t \, \rho(t,x) \dx \dt
+ \frac{\alpha}{2T}\sum_{i=1}^M  \norm{u_i}_{H^1(0,T)}^2
\end{equation}
with weighting parameters $c_1,c_2>0$ and a regularization parameter $\alpha>0$.
The first term in $J$ penalizes individuals remaining in the room at time~$T$.
The second term encourages individuals to leave the room as early as possible.
The last terms provides the required regularity for the control variables, 
so that the forward system \eqref{eq:forward_system} is well-defined.
From the modeling point of view, these terms also avoid unrealistic 
trajectories of the agents.
\item\textit{Optimal binding of a crowd}: 
In some applications it might be desired to keep the group of individuals 
together, \ie, trying to maintain a single group during an evacuation. 
This is also motivated by a similar approach which has been used to model the 
repulsive interaction of dogs in a flock of sheeps, see 
\cite{BurgerPinnauRothTotzeckTse:2016:1}. 
To this end, we define the center of mass and variance of $\rho$ as
\begin{equation*}
E_\rho(t) 
\coloneqq 
\frac{1}{M(t)}\int_\Omega x \rho(t,x) \dx
\quad 
\text{and} 
\quad 
V_\rho(t) 
\coloneqq
\frac{1}{M(t)}\int_\Omega \rho \abs{x-E_\rho(t)}^2 \dx
,
\end{equation*}
with total mass $M(t) = \int_\Omega \rho(t,x) \dx$.
A crowd is optimally kept together when the functional 
\begin{equation}
\label{eq:objective_minimal_variance}
J(\rho,\phi,\bx;\bu) 
\coloneqq
\frac{1}{2T} \int_0^T V_\rho(t) \dt + \frac\alpha{2T}\sum_{i=1}^M  
\norm{u_i}_{H^1(0,T)}^2
\end{equation}
is minimized.
\end{itemize}
Clearly, it is also possible to use a combination of the objective functionals 
\eqref{eq:objective_optimal_evacuation} and 
\eqref{eq:objective_minimal_variance}.

\subsection{Related Work}\label{sec:related}

We briefly review the literature regarding the analysis of the original Hughes 
model and related optimal control problems.

A first contribution on existence results for the Hughes model for the one-dimensional case is 
\cite{DiFrancescoMarkowichPietschmannWolfram:2011:1}. 
There it was shown, starting from the regularized version \eqref{eq:hreg}, that in the limit $\eps \to 0$ a suitable unique entropy solution~$\rho$ exists. 
The proof is based on a vanishing viscosity argument and Kruzkov's doubling of variables technique to show uniqueness. 
The results were complemented by more detailed studies on the unregularized problem, also numerically. 
For instance generalizations to higher spatial dimensions can be found in \cite{ColomboGaravelloLecureuxMercier:2012:1}, even for a slightly more general class of models. 
Further articles examine Riemann-type solutions to the unregularized problem; see \cite{AmadoriDiFrancesco:2012:1,ElKhatibGoatinRosini:2013:1,ElKhatibGoatinRosini:2014:1,DiFrancescoFagioliRosiniRussoAnd:2017:1} in one spatial dimension. 
As far as modeling is concerned, slightly different models were derived in \cite{BurgerDiFrancescoMarkowichWolfram:2014:1} based on a mean field games approach.
In \cite{CarrilloMartinWolfram:2016:1}, a modified approach using multiple local potentials $\phi_i$ instead of one global potential $\phi$ is introduced, removing the possibly unrealistic assumption that every pedestrian has complete information of the entire crowd. 
Moreover, in \cite{CarliniFestaSilvaWolfram:2016:1}, a discrete pedestrian model in a graph network is studied.

From a broader perspective, the optimal control of \eqref{eq:forward_system} 
falls into the class of the optimal control of coupled ODE-PDE systems.
Such problems, with models from a range of different applications have been 
analyzed, for instance, in 
\cite{ChudejPeschWaechterSachsBras:2009:1,WendlPeschRund:2010:1,KimmerleGerdts:2016:1,KimmerleGerdtsHerzog:2018:1,HoltmannspoetterRoeschVexler:2020:1}.
Also in the context of pedestrian dynamics a couple of contributions exist.
We want to mention \cite{AlbiBonginiCristianiKalise:2016:1} where a closely related microscopic model with crowds consisting of a (fixed) number of leaders and followers is studied. 
The interaction between individuals is a short range retraction and a long 
range attraction.
While leaders are not visible to the remaining crowd, they still influence it by taking part in these interactions. Then an optimal control problem arises as an external force acts on the leaders. 
The authors also consider, in the limit of many individuals (grazing interaction limit), macroscopic Boltzmann type equations for this interaction, while the number of leaders remains fixed and finite.
Similar studies can be found in \cite{BurgerPinnauRothTotzeckTse:2016:1} where external agents act as control. Again, they start from a microscopic ODE model and subsequently obtain a 
continuous model for the uncontrolled population by means of a mean field 
limit. For a general overview on interacting particle systems and control, we 
refer the reader to \cite{PinnauTotzeck:2018:1}.

Closer to our approach is the work of \cite{Borsche:2015:1}. There, a system of 
hyperbolic conservation laws for the density of different pedestrian groups is 
coupled to ODEs accounting for agents. Due to the low regularity of solutions 
to the hyperbolic equations, a regularization in the ODEs, similar to Lemma 
\ref{lem:ODE_reg} in our case, is used. See also 
\cite{BorscheKlarKuehnMeurer:2014:1,Borsche:2019:1} for a similar approach in 
different settings.

This paper is organized as follows. 
In \cref{sec:prelim} we collect the required notation, introduce some assumptions and state a precise existence and uniqueness result for the regularized system \eqref{eq:forward_system}. 
The full forward system involving also the ordinary differential equation \ref{eq:ODE} is investigated in \cref{sec:forward_system} and the linearized forward system in \cref{sec:linearized_system}.
The latter is required to establish the differentiability of the control-to-state map, which in turn is the basis of optimality conditions.
Based on this we discuss the optimal control problem in \cref{sec:opti} and derive first-order necessary optimality conditions. 
The presentation of numerical results will be postponed to a forthcoming publication.

\section{Mathematical Preliminaries}
\label{sec:prelim}

Let us first state the assumptions on the domain and data.
\begin{enumerate}[label=(A\arabic*)]
	\item 
		\label{assumption:A1}
		$\Omega \subset \R^2$ is an open, bounded domain with $C^4$-boundary $\partial\Omega$.
	\item 
		\label{assumption:A2}
		There exist two measurable sets $\partial\Omega_\textup{D}, \partial\Omega_\textup{W} \subset \partial\Omega$ \st $\partial\Omega_\textup{D} \cup \partial\Omega_\textup{W} = \partial\Omega$ and $\partial\Omega_\textup{D} \cap \partial\Omega_\textup{W} = \emptyset$.
		Moreover, $\partial\Omega_\textup{D}$ has positive measure with respect to the Lebesgue measure on $\partial\Omega$.

	\item 
		\label{assumption:A3}
		The initial density satisfies $\rho_0\in W^{3/2,4}(\Omega)$ and $0 \le \rho_0 \le 1$ \ale in $\Omega$.

	\item  
		\label{assumption:A4}
		There holds $f \in W^{3,\infty}(\R) \cap C_c(\R)$ with $f(0) = 1$, $f(1) = 0$ and $f(\rho)>0$ for all $\rho \in [0,1)$. 
\end{enumerate}

Moreover, we require some assumptions on the potential functions \eqref{eq:potential_term} of the agents, which depend on the kernel~$K$.
\begin{enumerate}[label=(K\arabic*)]
	\item
		\label{assumption:K1}
		The kernel $K \colon \R^2 \to \R$ is radially symmetric, \ie, $K(x) = k(\abs{x})$, where $k \colon [0,\infty) \to \R$ is nonnegative and decreasing.

	\item 
		\label{assumption:K2}
		The kernel satisfies $K \in W^{3,\infty}(\R^2)$.
\end{enumerate}

Finally, we consider an assumption on the velocity controls of the agents.
\begin{enumerate}[label=(C\arabic*)]
	\item
		\label{assumption:C1}
		There holds $u_i \in L^\infty(0,T;\R^2)$ and $\norm{u_i}_{L^\infty(0,T)} \le 1$ for  $i = 1, \ldots, M$.
\end{enumerate}
Throughout this article we frequently exploit the boundedness and Lipschitz continuity of the functions $K$, $f$, $h$ and $g(x) \coloneqq x\,f(x)$ and its derivatives. When doing so, we denote the bounds and Lipschitz constants by $C_K$, $C_f$, $C_h$ and $C_g$ respectively.\ $L_K$, $L_f$, $L_h$ and $L_g$.

\begin{remark}[Assumptions]
	\begin{enumerate}
		\item 
			Our results extend to the case $d=3$ upon adopting the Sobolev embeddings used in several places.
		\item 
			Assumption \cref{assumption:A2} essentially means that the door and wall parts of the boundary \eqq{do not meet}, \ie, we do not consider a truly mixed boundary value problem in order to avoid technical conditions ensuring sufficient regularity of solutions. 
			\cref{assumption:A2} can be replaced, \eg, in two spatial dimensions, by suitable angle conditions on the points where the two parts of the boundary meet.
			The interested reader is referred to results in \cite{Grisvard:2011:1}.
		\item 
			The optimal regularity for the initial datum $\rho_0$ is the Besov space $B_{pp}^{2-2/p}(\Omega)$; see \cite{DenkHieberPruess:2007:1}. 
		\item 
			While in Assumption \cref{assumption:A4} $f$ is defined on all of $\R$, as far as the modeling is concerned, only $\restr{f}{[0,1]}$ is relevant.
			Indeed, we will later see that the solution to \eqref{eq:forward_system1} satisfies $0\le \rho \le 1$.
		\item 
			A reasonable choice in \cref{assumption:K1} and \cref{assumption:K2} is the kernel function
			\begin{equation}
				\label{eq:kernel_function}
				K(x-x_i(t)) 
				= 
				\begin{cases}
				s\exp\paren[auto](){-\frac{\eta^2}{\eta^2-\abs{x-x_i(t)}^2}}, &\text{if}\ \abs{x-x_i(t)} < \eta,\\
				0, &\text{otherwise},
				\end{cases}
			\end{equation}
			where $s > 0$ is an intensity factor and $\eta > 0$ is related to an attraction radius.
	\end{enumerate}
\end{remark}

\textbf{Notations}. 
The space-time cylinder and its lateral surface are denoted by $Q_T \coloneqq (0,T) \times \Omega$ and $\Sigma_T = (0,T) \times \partial\Omega$, respectively.
The boundary surface can be divided into
\begin{equation*}
	\Sigma_{\textup{D},T} = (0,T) \times \partial\Omega_\textup{D}
	\text{ and }
	\Sigma_{\textup{W},T} = (0,T) \times \partial\Omega_\textup{W}
	.
\end{equation*}
The (Frobenius) inner product of two matrices $A, B \in \R^{n\times n}$ is denoted by
\begin{equation*}
	A \dprod B 
	\coloneqq
	\trace(A^\transp B)
	=
	\sum_{i,j=1}^n a_{ij}b_{ij}
	.
\end{equation*}
The Jacobian of a function $h \colon \R^2 \to \R^2$ is denoted by $Dh$ and the Hessian of a function $u \colon \R^2 \to \R$ is denoted by $\nabla^2u$.
Furthermore, $\eta_\gamma\in C^\infty_c(\R^2)$ is a standard mollifier, see, \eg, Ch.~4.4 in \cite{Brezis:2011:1}, \ie, a function satisfying
\begin{align}\label{eq:def_mollifier}
	\supp \eta_\gamma \subset \closure{B_{\gamma}(0)}
	\quad 
	\text{and}
	\quad 
	\int_{\R^2} \eta_\gamma \dx 
	= 
	1
	.
\end{align}
Note that 
\begin{equation}\label{eq:convergence_mollification}
	\norm{f - \eta_\gamma \ast f}_{C(\closure{\Omega})} 
	\to 
	0
	\text{ as }
	\gamma 
	\to 
	0
\end{equation} for every continuous function $f \in C(\closure{\Omega})$, cf. Prop.~4.21 in \cite{Brezis:2011:1}.

Finally, we introduce the subspace of $H^1(\Omega)$ incorporating the Dirichlet boundary conditions as
\begin{equation*}
	H^1_\textup{D}(\Omega)
	\coloneqq 
	\setDef{v \in H^1(\Omega)}{v = 0 \text{ \ale on } \partial\Omega_\textup{D}}
\end{equation*}
and, for $p\in [1,\infty]$, we denote the subspace of $W^{2,p}(\Omega)$ fulfilling the boundary conditions \eqref{eq:forward_bc} by
\begin{equation*}
	W^{2,p}_\text{ND}(\Omega)
	\coloneqq 
	\setDef{v\in W^{2,p}(\Omega)}{v = 0 \text{ \ale on } \partial\Omega_\textup{D}, \nabla v \cdot n = 0 \text{ \ale on } \partial\Omega_\textup{W}}
	.
\end{equation*}

For time-dependent functions we introduce, for $p \in (1,\infty)$ and $r,s \in \N_0$, the spaces
\begin{equation*}
	W^{r,s}_p(Q_T) 
	\coloneqq
	L^p(0,T;W^{r,p}(\Omega)) \cap W^{s,p}(0,T;L^p(\Omega))
	,
\end{equation*}
equipped with the natural norm $\paren[auto](){\norm{\cdot}_{L^p(0,T;W^{r,p}(\Omega))}^p + \norm{\cdot}_{W^{s,p}(0,T;L^p(\Omega))}^p}^{1/p}$.  
Spaces with non-integral $r$ and $s$ are defined, as usual, as (real) interpolation spaces. 
Of particular interest in our application is the space $W^{2,1}_p(Q_T)$ with $p > d = 2$, which fulfills the embedding
\begin{equation}\label{eq:w12q_embedding}
	W^{2,1}_p(Q_T) \embeds 
	C([0,T];W^{1,p}(\Omega)) \embeds 
	C(\closure{Q_T})
	.
\end{equation}
This is needed in order to allow point evaluations of the density~$\rho$, required in the ordinary differential equation \eqref{eq:forward_system3}. 
Finally, for functions from the Hölder space $C^{1,\alpha}(\Omega)$ we introduce the norm 
\begin{equation*}
	\norm{u}_{C^{1,\alpha}(\Omega)} 
	\coloneqq
	\norm{u}_{C^1(\Omega)} + \max _{\abs{\beta}=1} \abs{D^{\beta} u}_{C^{0, \alpha}(\Omega)}
	,
	\quad
	\text{where } 
	\abs{u}_{C^{0, \alpha}(\Omega)} 
	= 
	\sup _{x \neq y \in \Omega} \frac{\abs{u(x)-u(y)}}{\abs{x-y}^{\alpha}}.
\end{equation*}
In the following we collect some important properties of the function spaces used in this article. 
\begin{lemma}\label{lem:interpolation}
	For each $\theta \in [0,1]$, $p \in (1,\infty)$ and $0\le s < r$, the continuous embedding
	\begin{equation*}
		W^{1,p}(0,T;W^{s,p}(\Omega))\cap L^p(0,T;W^{r,p}(\Omega)) \embeds
		W^{\theta,p}(0,T; W^{\theta\,s + (1-\theta)\,r,p}(\Omega)).
	\end{equation*}
	holds.
\end{lemma}
\begin{proof}
	See Lemma~4.3 in \cite{DenkSaalSeiler:2008:1}.
\end{proof}

Similarly as above we define the Sobolev spaces
\begin{equation*}
	W^{r,s}_p(\Sigma_T) 
	\coloneqq
	L^p(0,T;W^{r,p}(\partial\Omega)) \cap W^{s,p}(0,T;L^p(\partial\Omega))
\end{equation*}
on the lateral boundary $\Sigma_T = (0,T) \times \partial\Omega$ of the space-time cylinder~$Q_T$.

The following trace theorem is proved in \cite[Lem.~3.5]{DenkHieberPruess:2007:1}; see also Sect.~2 in \cite{DenkPruessZacher:2008:1}:
\begin{lemma}
	\label{lem:trace_theorem}
	For $p>1$, the trace operators 
	\begin{align*}
		\gamma_0 
		&
		\colon W^{2,1}_p(Q_T) \to W^{2-1/p,1-1/(2p)}_p(\Sigma_T)
		,
		\\
		\gamma_1 
		&
		\colon W^{2,1}_p(Q_T) \to W^{1-1/p,1/2-1/(2p)}_p(\Sigma_T)
	\end{align*}
	defined by $\gamma_0 \rho = \restr{\rho}{\Sigma_T}$ and $\gamma_1 \rho = \nabla \rho \cdot n_{\Sigma_T}$
	are bounded and have a continuous right inverse. 
\end{lemma}
Recall the differential equation \eqref{eq:ODE}, where an extension to $\R^2$ of the density function~$\rho$ is used. 
For theoretical purposes we will use an extension operator fulfilling the following result from Lemma~6.37 in \cite{GilbargTrudinger:2001:1}:
\begin{lemma}\label{lem:extension} 
	Let $\alpha \in (0,1)$ be a fixed number. 
	There exists a linear, continuous extension operator 
	\begin{equation*}
		\E \colon C^{1,\alpha}(\closure{\Omega}) \to C^{1,\alpha}(\R^2)
	\end{equation*}
	such that
	\begin{equation*}
		\norm{\E f}_{C^{1,\alpha}(\R^2)} 
		\le 
		C_\E \, \norm{f}_{C^{1,\alpha}(\Omega)}
		\quad\text{and}\quad
		\norm{\E f}_{L^\infty(\R^2)} 
		\le 
		C_{\E,\infty}\norm{f}_{L^\infty(\Omega)},
	\end{equation*}
	holds for all $f \in C^{1,\alpha}(\closure{\Omega})$.
	For brevity we will write $\overline f \coloneqq\E f$.
\end{lemma}

\section{Analysis of the Forward System}
\label{sec:forward_system}

This section is devoted to showing the existence of strong solutions to the forward system \eqref{eq:forward_system} with boundary and initial conditions \eqref{eq:hreg_initial}, \eqref{eq:forward_bc}, \eqref{eq:ODE_init}. 
We proceed in two steps. 
First we provide auxiliary results on equation \eqref{eq:forward_system2} as well as on linear parabolic equations.
Then we prove existence of solutions to the complete forward system.

\subsection{Preliminary Results}

First, we study the regularized Eikonal equation \eqref{eq:forward_system2}.
\begin{lemma}\label{lem:ex_phi} 
	For given $\widetilde \rho \colon Q_T \to \R$ consider the equation 
	\begin{equation*}
		-\delta_1\,\laplace \phi + \abs{\nabla\phi}^2 
		= 
		\frac{1}{f(\widetilde \rho)^2+\delta_2}
		\quad
		\text{in } Q_T
		,
	\end{equation*}
	with boundary conditions \eqref{eq:forward_bc}. 
	We have:
	\begin{enumerate}
		\item 
			If $\widetilde \rho \in C([0,T];L^2(\Omega))\cap H^1(0,T;H^1(\Omega)^*)$, then there exists a unique strong solution which satisfies $\phi \in L^\infty(0,T;W^{2,p}_\text{ND}(\Omega))\cap H^1(0,T;H^1(\Omega))$ for all $2 \le p < \infty$.
			Moreover, the a~priori estimates
			\begin{equation}\label{eq:phi_apriori}
				\norm{\phi}_{L^\infty(0,T;W^{2,p}(\Omega))}
				\le 
				\widetilde C_\phi
				\quad 
				\text{and}
				\quad
				\norm{\phi}_{H^1(0,T;H^1_\textup{D}(\Omega))}
				\le 
				\widetilde C_\phi\,\norm{\widetilde \rho}_{H^1(0,T;H^1(\Omega)^*)}
			\end{equation}
			hold with a positive constant $\widetilde C_\phi$ depending on $p,\delta_1,\,\delta_2,\,\Omega$ and $T$ only.
		\item 
			If $\widetilde \rho \in W^{2,1}_p(Q_T)$ holds, the strong solution $\phi$ additionally belongs to $W^{4,1}_p(Q_T)$ and satisfies the a~priori estimate
			\begin{align}\label{eq:phi_apriori2}
				\norm{\phi}_{W^{1,p}(0,T;W^{2,p}(\Omega))} 
				\le 
				\overline C_\phi\,\norm{\widetilde \rho}_{W^{2,1}_p(Q_T)}
				,
			\end{align}
			with a constant $\overline C_\phi$ depending on $p,\, \delta_1,\,\delta_2,\,\Omega$ and $T$ only.

		\item 
			For any $\widetilde \rho_1, \widetilde \rho_2 \in C([0,T];L^2(\Omega))$, the corresponding solutions $\phi_1$ and $\phi_2$ satisfy the Lipschitz estimate
			\begin{align}\label{eq:phi_stability}
				\norm{\phi_1(\cdot,t)  -\phi_2(\cdot,t)}_{W^{2,2}(\Omega)} 
				\le 
				\widehat{C}_\phi \, \norm{\widetilde \rho_1(\cdot, t) - \widetilde \rho_2(\cdot,t)}_{L^2(\Omega)}
			\end{align}
			for all $t\in [0,T]$, with $\widehat{C}_\phi$ depending on $p,\, \delta_1,\, \delta_2,\,\Omega$ and $T$ only.
	\end{enumerate} 
\end{lemma}
From now on we shall use the definition $C_\phi \coloneqq \max\paren[big]\{\}{\widetilde{C_\phi}, \overline C_\phi, \widehat C_\phi}$.
\begin{proof}
	We first show assertion (i). Note that due to the continuity of $\widetilde \rho$ in time it makes sense to define, for fixed $t \in [0,T]$, the function
	\begin{equation}\label{eq:def_qt}
		q_t(x) 
		\coloneqq
		\frac{1}{\delta_1^2} \frac{1}{f(\widetilde \rho(t,x))^2 + \delta_2}
		.
	\end{equation}
	By assumption \cref{assumption:A4} on $f$ we then have $q_t \in L^\infty(\Omega)$.

	\underline{Step 1: Existence.}  
	First note that an application of the transformation 
	\begin{align}\label{eq:psitrans}
		\psi(t,\cdot) = \e^{-\phi(t,\cdot)/\delta_1} - 1
	\end{align}
	to \eqref{eq:forward_system2} yields for all $t\in [0,T]$ the linear problem
	\begin{equation}\label{eq:psi}
		\begin{aligned}
			-\laplace \psi(t,\cdot) + q_t\,\psi(t,\cdot)
			&
			= 
			-q_t(\cdot)
			&
			&
			\text{in } \Omega
			,
			\\
			\psi(t,\cdot)
			&
			= 
			0 
			&
			&
			\text{on }\partial\Omega_\textup{D}
			,
			\\
			\partial_n  \psi(t,\cdot)
			&
			= 
			0 
			&
			&
			\text{on } \partial\Omega_\textup{W}
			.
		\end{aligned}
	\end{equation}
	As $q_t \in  L^\infty(\Omega)$ and $q_t > 0$ \ale in $\Omega$, the Lax-Milgram lemma yields for every $t\in [0,T]$ the existence of a unique weak solution $\psi(t,\cdot) \in H^1_\textup{D}(\Omega)$ and the estimate 
	\begin{equation*}
		\norm{\psi(t,\cdot)}_{H^1(\Omega)} 
		\le 
		C \, \norm{q_t}_{L^2(\Omega)}
		.
	\end{equation*}
	Elliptic regularity theory (see for instance Theorem~3.17 in \cite{Troianiello:1987:1}) then implies $\psi(t,\cdot) \in W^{2,q}(\Omega)$ for any $2 \le q < \infty$, and there holds
	\begin{equation}\label{eq:psi_apriori}
		\norm{\psi(t,\cdot)}_{W^{2,q}(\Omega)} 
		\le 
		C \, \norm{q_t}_{L^q(\Omega)} 
		\eqqcolon 
		C_\phi		
	\end{equation}
	for $t\in [0,T]$.
	For given $x \in \Omega$, the function $t \mapsto q_t(x)$ belongs to $L^\infty(0,T)$ again by \cref{assumption:A4}.
	Taking the supremum over $0 \le t \le T$ yields the regularity $\psi \in L^\infty(0,T;W^{2,q}(\Omega))$. 
	Formally differentiating equation \eqref{eq:psi} with respect to time we see that $\partial_t \psi$ satisfies 
	\begin{equation*}
		-\laplace \partial_t\psi(t,\cdot) + q_t\partial_t\psi(t,\cdot) 
		= 
		-(1+\psi(t,\cdot))(\partial_t q_t) 
		\quad
		\text{in } \Omega
		.
	\end{equation*}
	Using its definition and the fact that $\widetilde \rho \in H^1(0,T;H^1(\Omega)^*)$, the derivative of $q_t(x)$ with respect to time is an element of $L^2(0,T;H^1(\Omega)^*)$.
	Therefore, Sect.~2.2.2, Corollary, p.99 in \cite{Troianiello:1987:1} (after replacing $\partial_t q_t(x)$ by a continuous in time approximation and passing to the limit) yields the existence of $\partial_t \psi \in L^2(0,T;H^1(\Omega))$ with
	\begin{align}\label{eq:psi_apriori_time}
		\norm{\partial_t \psi }_{L^2(0,T;H^1(\Omega))} 
		\le 
		C \, \norm{\partial_t \widetilde \rho}_{L^2(0,T;H^1(\Omega)^*)}
		.
	\end{align}	

	\underline{Step 2: Strict lower bound.}
	In order to invert \eqref{eq:psitrans} and obtain a solution to \eqref{eq:forward_system2}, we need to ensure that $\psi > -1$ holds. 
	This follows from Theorem~4 in \cite{LeSmith:2002:1}, provided that we can show $\psi \ge -1$, $\norm{\psi}_{L^\infty(Q_T)} < \infty$ and Hölder continuity of $\psi(t,\cdot)$ for \aa $t\in(0,T)$.
	The last two assertions are a direct consequence of the embeddings $W^{2,p}(\Omega) \embeds L^\infty(\Omega)$ and $W^{2,p}(\Omega) \embeds C^{0,1/2}(\Omega)$ as $p \ge 2$ and we are in dimension~2, combined with \eqref{eq:psi_apriori}. 
	To show the first part we choose $\phi = (\psi+1)_-$, \ie, the negative part of $\psi+1$, as test function in the weak formulation of \eqref{eq:psi} and obtain
	\begin{equation*}
		\int_\Omega \abs{\nabla  (\psi(t,\cdot)+1)_-}^2 \dx + \int_\Omega q_t(x)(\psi + 1)_-^2 \dx 
		= 
		0
		.
	\end{equation*}
	Since $q_t$ is strictly positive, this implies that $(\psi + 1)_- = 0$ holds \ale in $\Omega$. 
	Thus there exists a positive constant $C_\psi$ \st
	\begin{equation*}
		\psi  \ge C_\psi 
		> 
		- 1
		\text{ \ale in } \Omega
		.
	\end{equation*}
	We can therefore invert the transformation \eqref{eq:psitrans} and conclude 
	that the solution of \eqref{eq:forward_system2}, \eqref{eq:forward_bc} 
	fulfills the desired regularity, as the regularity is unaffected by the 
	transformation (using $\psi \in L^\infty(0,T;L^\infty(\Omega))$ and $\psi > 
	-1$). This ends the proof of assertion (i). To conclude (ii) we only have 
	to show the additional regularity $\partial_t \psi \in L^p(0,T;W^{2,p}(\Omega))$ which
	follows under the assumption $\widetilde \rho \in W^{2,1}_p(Q_T)$ and standard elliptic
	regularity theory, see \cite[Theorem 3.17]{Troianiello:1987:1}.

Similarly, in case $\widetilde \rho \in W^{2,1}_p(Q_T)$ we obtain $\partial_t \psi \in L^p(0,T;W^{2,p}(\Omega))$, which gives the desired estimate.

	\underline{Step 3: Lipschitz estimate.} 
	To show (iii), denote by $q_{t,1}$, $q_{t,2}$ the respective coefficients for $\widetilde \rho_1$ and $\widetilde \rho_2$ as in \eqref{eq:def_qt} and, analogously, let $\psi_1$ and $\psi_2$ be the respective solutions to \eqref{eq:psi}. 
	Then $\overline \psi = \psi_1 - \psi_2$ satisfies
	\begin{equation}\label{eq:psi_diff}
		\begin{aligned}
			-\laplace \overline \psi(t,\cdot) + q_{t,1}(x) \, \overline \psi(t,\cdot) 
			&
			=
			(1 -\psi_2(t,\cdot)) \, \overline q_t(x) 
			&
			&
			\text{in } \Omega
			,
			\\
			\overline \psi(t,\cdot) 
			&
			= 
			0 
			&
			&
			\text{on }\partial\Omega_\textup{D}
			,
			\\
			\partial_n \overline \psi(t,\cdot) 
			&
			= 
			0 
			&
			&
			\text{on } \partial\Omega_\textup{W}
			,
		\end{aligned}
	\end{equation}
	for $t\in [0,T]$ with $\overline q_t = q_{t,1} - q_{t,2}$. 
	Noting that $q_t$ is Lipschitz continuous as a function of $\widetilde \rho$ (due to \cref{assumption:A4}) and applying the a~priori estimate \eqref{eq:psi_apriori} to \eqref{eq:psi_diff} yields
	\begin{align*}
		\norm{\overline \psi(t,\cdot)}_{W^{2,2}(\Omega)} 
		\le 
		C \, \norm{\overline q_t}_{L^2(\Omega)} 
		\le  
		C \, \norm{\widetilde \rho_1(t,\cdot) - \widetilde \rho_2(t,\cdot)}_{L^2(\Omega)}
		,
	\end{align*}
	where we used the boundedness of $(1-\psi_2(t,\cdot))$ in $L^\infty(\Omega)$ and, again, $W^{2,p}(\Omega) \embeds L^\infty(\Omega)$ and \eqref{eq:psi_apriori}. This implies (iii) and completes the proof.
\end{proof}

To obtain the desired $W^{2,1}_p(Q_T)$-regularity of the density function~$\rho$ we will need the following lemma taken from \cite{DenkHieberPruess:2007:1} but adopted to our notation.
\begin{lemma}\label{lem:W2p_regularity} 
	Let assumptions \cref{assumption:A1}--\cref{assumption:A4} hold, $p\in [2,\infty)$ and $\eps > 0$. 
	Suppose that $c \in L^p(Q_T)$, $b \in L^p(0,T;L^\infty(\Omega))$, $r \in W^{1-1/p,1/2-1/(2p)}_p(\Sigma_T)$ and $\rho_0\in W^{2-2/p,p}(\Omega)$ are given.
	Then the problem
	\begin{equation*}
		\begin{aligned}
			\partial_t \rho - \eps \laplace \rho + b\cdot \nabla \rho 
			&
			= c 
			&
			&
			\text{in } Q_T
			,
			\\
			\eps \nabla \rho \cdot n 
			&
			= r 
			&
			&
			\text{on }\Sigma_{\textup{W},T}
			,
			\\
			\rho 
			&
			= 
			0 
			&
			&
			\text{on }\Sigma_{\textup{D},T}
			,
			\\
			\rho(0) 
			&
			=
			\rho_0  
			&
			&
			\text{in } \Omega
		\end{aligned}
	\end{equation*}
	admits a unique strong solution $\rho \in W^{2,1}_p(Q_T)$ depending continuously on the input data $c$, $r$ and~$\rho_0$.
\end{lemma}
\begin{proof}
	See Theorem~2.1 in \cite{DenkHieberPruess:2007:1}.
\end{proof}

Finally, we need the following regularity result for $p=2$ with flux boundary conditions.
\begin{lemma}\label{lem:H2_reg}
	Given $h \in H^1(0,T;H^1(\Omega)) \cap L^\infty(Q_T)$, $g \in L^\infty(\R)$ and $\rho_0 \in H^1(\Omega)$, the variational problem
	\begin{multline}\label{eq:rho_h_fixed}
		\int_{Q_T} \partial_t \rho\,\xi\dx\dt
		+ \eps\int_{Q_T} \nabla\rho\cdot \nabla \xi \dx\dt
		- \int_{Q_T} g(\rho)\,h\cdot \nabla \xi\dx\dt
		\\
		=
		-\eta \int_{\Sigma_T} \chi_{\partial\Omega_\textup{D}} \, \rho \, \xi \ds_x \dt 
		\quad 
		\text{for all } \xi \in L^2(0,T;H^1(\Omega))
	\end{multline}
	has a unique solution $\rho \in L^\infty(0,T;H^1(\Omega))\cap H^1(0,T;L^2(\Omega))$ with $\rho(0) = \rho_0$.
\end{lemma}
The proof mainly uses standard methods but since, to the best of the authors' knowledge, a proof matching our boundary conditions is not available in the literature, we included it into \cref{section:auxiliary_results}.

Next we define our notion of solution for the ODE \eqref{eq:forward_system3}. 
\begin{definition}\label{def:ODE_extended} 
	Fix $2 < p < \infty$ and $\rho \in W^{2,1}_p(Q_T)$. 
	Then for given $f$ and $u$ satisfying assumptions \cref{assumption:A4} and \cref{assumption:C1} and an initial value $x_0\in \R^2$ we say that $x \colon [0,T] \to \R^2$ is a solution to
	\begin{equation}\label{eq:def_generalized_ode}
		\dot x(t) 
		= 
		f\paren[big](){\overline\rho(x(t),t)} \, u(t),
	\end{equation}
	if it is absolutely continuous, satisfies \eqref{eq:def_generalized_ode} for \aa $t \in [0,T]$
	and $x(0) = x_0$.
\end{definition}

We have the following result about the existence of a solution of \eqref{eq:def_generalized_ode}.
\begin{lemma}\label{lem:ode_existence} 
	For given $2 <p < 4$, $u\in L^\infty(0,T;\R^2)$ satisfying assumption \cref{assumption:C1} and $\rho \in W^{2,1}_p(Q_T)$, there exists a unique, absolutely continuous solution $x \colon [0,T] \to \R^2$ to \eqref{eq:def_generalized_ode} satisfying $x(0) = x_0$. 
	Furthermore, $x \in W^{1,\infty}(0,T)$ holds.
	For $\rho_1,\,\rho_2 \in W^{2,1}_p(Q_T)$, the corresponding solutions $x_1$ and $x_2$ satisfy
	\begin{equation}\label{eq:ODE_stab}
		\norm{x_1-x_2}_{L^\infty(0,T;\R^2)} 
		\le 
		C_s\,\norm{\rho_1 - \rho_2}_{L^\infty(Q_T)}
		,
	\end{equation}
	where the constant $C_s$  depends on $T$, $C_\E$, $C_{\E,\infty}$ and the Lipschitz constants of $f$ and $\rho$.	
\end{lemma}
\begin{proof}
	First note that $\rho \in L^p(0,T;W^{2,p}(\Omega)) \embeds L^p(0,T;C^{1,\alpha}(\Omega))$, for some $\alpha > 0$, so that the application of the extension operator from \cref{lem:extension} is well-defined. 
	Since also $f$ is Lipschitz continuous with Lipschitz constant $L_f$ by assumption \cref{assumption:A4}, the function $f(\overline\rho(t,x))$ satisfies the Carathéodory conditions (see \cref{def:Caratheodory} in the Appendix) and thus there exists a solution in the sense of \cref{def:ODE_extended}; see Ch.~I, Theorem~5.1 in \cite{Hale:2009:1}.
	Furthermore, as we also have $\overline\rho  \in L^p(0,T;C^{0,1}(\R^2))$, we obtain with $C_\rho(t) = \norm{\rho(t)}_{W^{2,p}(\Omega)}$ and the property $C_\rho\in L^p(0,T)$ that 
	\begin{align}
		\MoveEqLeft
		\abs{f(\overline\rho(t,x)) - f(\overline\rho(t,y))}
		\le 
		L_f\,\abs{\overline\rho(t,x)-\overline\rho(t,y)}
		\le 
		L_f\,\norm{\overline \rho(t)}_{C^{1,\alpha}(\R^2)}\,\abs{x-y}
		\nonumber
		\\ 
		&
		\le 
		L_f\,C_E\,\norm{\rho(t)}_{C^{1,\alpha}(\Omega)}\,\abs{x-y}
		\nonumber
		\\ 
		&
		\le L_{f,\E,\rho}(t)\,\abs{x-y}
		\quad\text{ for \aa\ } t \in [0,T] \text{ and all } x, y \in \R^2,
		\label{eq:ODE_Lipschitz2}
	\end{align}
	where $L_{f,\E,\rho}(t) \coloneqq L_f\,C_E\,C_\infty\,C_\rho(t)$ and with $C_\infty$ the embedding constant for $W^{2,p}(\Omega) \embeds C^{1,\alpha}(\overline\Omega)$.
	The estimate \eqref{eq:ODE_Lipschitz2} implies uniqueness by Ch.~I, Thm.~5.3 in \cite{Hale:2009:1}. The additional regularity $x \in W^{1,\infty}(0,T;\R^2)$ is a consequence of the boundedness in $L^\infty(0,T;\R^2)$ of the right-hand side of \eqref{eq:def_generalized_ode}. 
	To establish the stability estimate we show
	\makeatletter
	\ltx@ifclassloaded{numapde-preprint}{%
		}{
	}
	\makeatother
	\begin{align*}
		\MoveEqLeft
		\abs{x_1(t) - x_2(t)} 
		\le 
		\int_0^t \abs{f(\overline\rho_1(s,x_1(s))) - f(\overline\rho_2(s,x_2(s)))}\,\abs{u(s)}\ds 
		\\
		&
		\le 
		\int_0^t \abs{f(\overline\rho_1(s,x_1(s))) - f(\overline\rho_1(s,x_2(s)))} \ds 
		+\int_0^t \abs{f(\overline\rho_1(s,x_2(s))) - f(\overline\rho_2(s,x_2(s)))}\ds
		\\
		&
		\le 
		\int_0^t L_{f,\rho,\E}(s)\,\abs{x_1(s) - x_2(s)}\ds + L_f \, C_\E \, C_{\E,\infty}\, \,t\,\norm{\rho_1 - \rho_2}_{L^\infty(Q_t)},
	\end{align*}
	where we used that the extension $\E$ is also continuous with respect to the $L^\infty$-norm. 
	An application of Gronwall's inequality in integral form then yields, for $t \in (0,T)$,
	\begin{align*}
		\abs{x_1(t) - x_2(t)} 
		&
		\le 
		L_f \, C_\E \, C_{\E,\infty}\,t\,\norm{\rho_1 - \rho_2}_{L^\infty(Q_T)} \, \exp\paren[auto](){\int_0^t L_{f,\rho,\E}(r) \d r}
		\\
		&
		\le 
		C_s \paren[auto](){L_f, L_{f,\rho,\E}, C_\E, C_{\E,\infty}, T} \, \norm{\rho_1 - \rho_2}_{L^\infty(Q_T)}
		.
	\end{align*}
\end{proof}

Next, we state an existence and stability result for a regularized version of \eqref{eq:def_generalized_ode}. 
Note that the following result requires less regularity for the density function $\rho$.
\begin{lemma}\label{lem:ODE_reg} 
	Fix $2 <p < \infty$ and $\rho \in C([0,T];L^2(\Omega))$. 
	Then for given $u\in  L^\infty(0,T)$, $f$ satisfying assumption \cref{assumption:A4} and \cref{assumption:C1}, and every $\gamma > 0$, there exists, a unique, absolutely continuous solution $x \colon [0,T]\to \R^2$ to
	\begin{equation}\label{eq:ode_reg}
		\dot x(t) 
		= 
		f\paren[auto](){(\overline{\eta_\gamma \ast \rho})(t,x(t))} \, u(t)
	\end{equation}	
	satisfying $x(0) = x_0 \in \R^2$ and $x \in W^{1,\infty}(0,T)$. 
	Here, $\eta_\gamma$ is a standard mollifier as in \eqref{eq:def_mollifier} and $\ast$ denotes the convolution \wrt to the $x$-variable. 
	Furthermore, for $\rho_1,\,\rho_2 \in C([0,T];L^2(\Omega))$, the corresponding solutions $x_1$ and $x_2$ satisfy
	\begin{align}\label{eq:ODE_stab2}
		\abs{x_1(t)-x_2(t)} 
		\le 
		C_s \,\norm{\eta_\gamma}_{L^2(\Omega)}\,\norm{\rho_1 - \rho_2}_{L^\infty(0,T;L^2(\Omega))}
		,
	\end{align}
	where the constant $C_s$ depends on $T$, $C_\E$, $C_{\E,\infty}$ and the Lipschitz constants of $f$ and $\rho$.
\end{lemma}
\begin{proof}
	As $f$ and $\overline{\eta_\gamma \ast \rho}$ are Lipschitz continuous (see assumption \cref{assumption:A4} and \cref{lem:conv_lipschitz}), the function $f\paren[auto](){(\overline{\eta_\gamma \ast \rho})(t,x)}$ satisfies the Carathéodory conditions of \cref{def:ODE_extended} and thus there exists a solution in the sense of \cref{def:ODE_extended}; see Ch.~I, Thm.~5.1 in \cite{Hale:2009:1}. 
	In particular there exists a positive function $L_{f,\gamma,\rho,\E} \in L^p(0,T)$ such that
	\begin{equation}\label{eq:ODE_Lipschitz}
		\abs{f((\overline{\eta_\gamma \ast \rho})(t,x)) - f((\overline{\eta_\gamma \ast \rho})(t,y))} 
		\le 
		L_{f,\gamma, \rho,\E}(t)\,\abs{x-y} 
	\end{equation}
	for \aa $t \in [0,T]$ and all $x, y \in \Omega$.
	This implies uniqueness by Ch.~I, Thm.~5.3 in \cite{Hale:2009:1}. To show the stability estimate we apply \cref{lem:ode_existence} with $\rho_1$ and $\rho_2$ replaced by $\eta_\gamma \ast \rho_1$ and $\eta_\gamma \ast \rho_2$, respectively, which yields
	\begin{align*}
		\norm{x_1-x_2}_{L^\infty(0,T)} 
		\le 
		C_s\,\norm{\eta_\gamma \ast (\rho_1 - \rho_2)}_{L^\infty(Q_T)}
	\end{align*}
	with $C_s$ defined as in the proof of Lemma~\ref{lem:ode_existence}.
	Applying Young's inequality for convolutions to the norm on the right-hand side allows us to estimate
	\begin{equation*}
		\norm{\eta_\gamma \ast (\rho_1 - \rho_2)}_{L^\infty(Q_T)}
		\le 
		\norm{\eta_\gamma}_{L^2(\Omega)}\,\norm{\rho_1 - \rho_2}_{L^\infty(0,T;L^2(\Omega))}
		,
	\end{equation*}
	which completes the proof. 
\end{proof}

We continue with the following Lipschitz estimate for the transport term in \eqref{eq:forward_system1}, which is needed in the theorem right after the next.
\begin{lemma}\label{lem:beta_lipschitz}
	Given $\rho_1,\,\rho_2 \in C([0,T];H^1(\Omega))$, $\phi_1,\,\phi_2 \in L^\infty(0,T;W^{2,p}(\Omega))$ and $\bx_1,\,\bx_2 \in L^\infty(0,T;\R^2)$ the function $\beta(\rho,\phi,\bx)$ satisfies the Lipschitz inequality
	\begin{align*}
		\MoveEqLeft
		\norm{\rho_1\,\beta(\rho_1,\phi_1,\bx_1) - \rho_2\,\beta(\rho_2,\phi_2,\bx_2)}_{L^2(Q_T)} 
		\\
		&
		\le 
		L_\beta \paren[auto](){ \norm{\rho_1-\rho_2}_{L^2(Q_T)} + \norm{\nabla \phi_1 - \nabla \phi_2}_{L^2(Q_T)} + \norm{\bx_1 - \bx_2}_{L^\infty(0,T;\R^2)^M}},
	\end{align*}
	where $L_\beta$ depends on the Lipschitz and boundedness constants of $f$, $h$ and $K$.
\end{lemma}
\begin{proof}
	We define $g(\rho) = \rho\,f(\rho)$ and write $\rho_i\,\beta(\rho_i,\phi_i,\bx_i) = g(\rho_i) \, h\paren[big](){\nabla(\phi_i+\phi_K(\bx_i;\cdot))}$, $i=1,2$. Using the Lipschitz continuity of $g$, $h$ and $\phi_K$ we obtain
	\begin{align*}
		\MoveEqLeft
		\norm[big]{\rho_1\,\beta(\rho_1,\phi_1,\bx_1) - \rho_2\,\beta(\rho_2,\phi_2,\bx_2)}_{L^2(0,T;L^2(\Omega))}
		\\
		&
		\le 
		L_g\,\norm{\rho_1-\rho_2}_{L^2(Q_T)}\,\norm[big]{h\paren[big](){\nabla(\phi_1+\phi_K(\bx_1;\cdot))}}_{L^\infty(Q_T)}
		\\
		&
		\quad
		+ C_g \, L_h \paren[auto](){\norm{\nabla(\phi_1-\phi_2)}_{L^2(Q_T)} + L_{\phi_K} \norm{\bx_1-\bx_2}_{L^\infty(0,T;\R^2)^M}}
		.
	\end{align*}
\end{proof}

\subsection{Existence for the Full Forward System}
We are now in a position to show the following existence and uniqueness result.
\begin{theorem}\label{thm:ex_state} 
	Let assumptions \cref{assumption:A1,assumption:A2,assumption:A3,assumption:A4}, \cref{assumption:K1,assumption:K2,assumption:C1} hold and fix $2 <p < \infty$. 
	Then for any given control $\bu = (u_1,\ldots, u_M)^\transp \in L^\infty(0,T;\R^2)^M$ and any $T>0$, there exists a unique solution $(\rho,\phi,\bx)$ to \eqref{eq:forward_system} with initial and boundary conditions \eqref{eq:hreg_initial}, \eqref{eq:forward_bc}, \eqref{eq:ODE_init} which satisfies $\rho \in W^{2,1}_p(Q_T)$, $\phi\in L^\infty(0,T;W^{2,p}(\Omega))$ and $\bx$ is a solution to \eqref{eq:forward_system3} in the sense of \cref{def:ODE_extended}. 
	Moreover, the a~priori estimate
	\begin{equation*}
		\norm{\rho}_{W^{2,1}_p(Q_T)}
		+ \norm{\phi}_{L^\infty(0,T;W^{2,p}(\Omega))} 
		\le 
		C \norm{\rho_0}_{W^{1,p}(\Omega)},
	\end{equation*}
	holds with $C$ depending only on the domain, the bounds for the coefficients and the respective kernel.
\end{theorem}

The structure of \eqref{eq:forward_system1}--\eqref{eq:forward_system2} is very similar to chemotaxis models with volume filling, see for instance \cite{PainterHillen:2002:1}, except for the additional nonlinearity of the squared gradient term in \eqref{eq:forward_system2}, which can, however, be handled using \cref{lem:ex_phi}. 
Therefore, the existence and uniqueness of solutions can be proved using Banach's fixed point theorem, similar to, \eg, Thm.~3.1 in \cite{EggerPietschmannSchlottbom:2015:1}. 
The main issue in our situation is the additional coupling to the system of ODEs \eqref{eq:forward_system3}, which requires $\rho$ to be regular enough to allow point evaluations. 
Our strategy is to introduce an additional regularization in \eqref{eq:forward_system3} in order to be able to perform the fixed point argument in the relatively \eqq{large} space $C([0,T];L^2(\Omega))$. 
We then show additional regularity and pass to the limit to recover the original system.

\begin{proof}
	The proof consists of two parts. 
	First we show existence with \eqref{eq:forward_system3} replaced by the regularized version \eqref{eq:ode_reg}. Then, we perform the limit $\gamma \to 0$ to recover the original problem. 

	\underline{Step 1: Fixed point argument}:
	We consider the fixed-point operator
	\begin{align}
		S \colon C([0,T];L^2(\Omega)) \to C([0,T];L^2(\Omega))
		,
		\quad 
		\widetilde \rho \mapsto \rho_\gamma
		,
	\end{align}
	where, for fixed $\gamma > 0$, $\rho_\gamma$ is the unique weak solution to the system
	\begin{alignat}{2}
		\label{eq:rho_lin}
		\partial_t\rho_\gamma - \nabla\cdot\paren[big](){\widetilde\rho\,\beta(\widetilde \rho, \phi_\gamma, \bx_\gamma)} 
		&
		=
		\eps\,\laplace \rho_\gamma 
		&
		&
		\quad 
		\text{in } Q_T,
		\\
		\label{eq:phi_reg_lin}
		- \delta_1\,\laplace \phi_\gamma + \abs{\nabla \phi_\gamma}^2 
		&
		= \frac{1}{f(\widetilde \rho)^2+\delta_2} 
		&
		&
		\quad 
		\text{in } Q_T 
		,
		\\ 
		\label{eq:ODE_reg}
		\dot x_{i,\gamma}(t) 
		&
		= 
		f\paren[big](){(\overline{\eta_\gamma \ast\widetilde\rho})(t,x_{i,\gamma}(t))}\,u_i(t),
		&
		& 
		\quad 
		\text{for } t \in [0,T],
		\quad 
		i=1,\ldots, M
		\intertext{subject to the boundary conditions}
		\varepsilon\,\nabla\rho_\gamma \cdot n
		&
		= 
		- \widetilde\rho\,\beta(\widetilde\rho,\phi_\gamma,\bx_\gamma)\cdot n - \eta\,\rho_\gamma\,\chi_{\partial\Omega_\textup{D}} 
		&
		&
		\quad
		\text{on } \Sigma_T
		.
		\nonumber
	\end{alignat}
	The boundary conditions for $\phi$ and the initial conditions for all variables are as in \eqref{eq:forward_bc} and \eqref{eq:hreg_initial}, \eqref{eq:ODE_init}, respectively.
	The equation for $\rho_\gamma$ is understood in the weak sense, \ie,
	\begin{equation}
		\label{eq:weak_form_rho_equation_reg}
		\int_\Omega \partial_t \rho_\gamma\,\xi \dx + \varepsilon \int_\Omega \nabla \rho_\gamma\cdot \nabla \xi \dx
		= 
		- \int_\Omega \widetilde\rho\,\beta(\widetilde\rho,\phi_\gamma,\bx_\gamma)\cdot\nabla \xi \dx
		- \eta \int_{\partial\Omega_\textup{D}} \rho_\gamma \, \xi \ds_x
	\end{equation}
	for all $\xi\in H^1(\Omega)$ and \aa $t\in [0,T]$. 
	Applying \cref{lem:H2_reg} with $g(\rho) = \rho f(\rho)$ and $h = h\paren[auto](){\nabla (\phi_\gamma + \phi_K(\bx_\gamma;\cdot))}$ shows that there exists a unique weak solution 
	\begin{equation*}
		\rho_\gamma \in L^\infty(0,T;H^1(\Omega))\cap H^1(0,T;L^2(\Omega)) \embeds 
		C([0,T];L^2(\Omega))
		.
	\end{equation*}
	Choosing $\rho$ itself as a test function in the weak formulation \eqref{eq:weak_form_rho_equation_reg} yields, after an application of the weighted Young's inequality with parameter $\eps/ 2$, the trace inequality for $H^1$-functions, and an integration in time, the a~priori estimate
	\begin{align*}
		\norm{\rho_\gamma}_{C([0,T];L^2(\Omega))}
		&
		\le 
		C \paren[big](){\norm{\widetilde \rho\,\beta(\widetilde \rho,\phi_\gamma,\bx_\gamma)}_{L^2(0,T; L^2(\Omega))} + \norm{\rho_0}_{L^2(\Omega)}} 
		\\
		&
		\le 
		C \paren[auto](){\sqrt{T\,\abs{\Omega}} +  \norm{\rho_0}_{L^2(\Omega)}}
		,
	\end{align*}
	where we used $\abs{\tilde \rho\,\beta(\tilde \rho,\phi_\gamma, \bx_\gamma)} \le 1$.

	Next, we show that $S$ is a contraction. 
	Let $\widetilde \rho_1, \, \widetilde \rho_2 \in C([0,T];L^2(\Omega))$ be arbitrary.
	The corresponding densities, potentials and agent positions are denoted by $\rho_{\gamma,1}$, $\rho_{\gamma,2}$, $\phi_{\gamma,1}$, $\phi_{\gamma,2}$ and $\bx_{\gamma,1}$, $\bx_{\gamma,2}$, respectively.
	Using \cref{lem:beta_lipschitz} we have
	\begin{multline}
		\label{eq:rho_stab_full}
		\norm{\rho_{\gamma,1}-\rho_{\gamma,2}}_{C([0,T];L^2(\Omega))}
		\le
		\norm{\widetilde \rho_1\,\beta(\widetilde \rho_1,\phi_{\gamma,1},\bx_{\gamma,1})
		-\widetilde \rho_2\,\beta(\widetilde \rho_2,\phi_{\gamma,2},\bx_{\gamma,2})}_{L^2(0,T;L^2(\Omega))} 
		\\
		\le
		L_\beta \paren[big](.{\norm{\widetilde \rho_1-\widetilde\rho_2}_{L^2(0,T;L^2(\Omega))} + \norm{\nabla (\phi_{\gamma,1} - \phi_{\gamma,2})}_{L^2(0,T;L^2(\Omega))}}
		+ 
		\paren[big].){\norm{\bx_{\gamma,1} - \bx_{\gamma,2}}_{L^\infty(0,T)^M}}
		.
	\end{multline}
	Applying \cref{lem:ex_phi,lem:ODE_reg} to the second and third terms on the right-hand side, respectively, yields
	\begin{equation*}
		\norm{\rho_{\gamma,1}-\rho_{\gamma,2}}_{C([0,T];L^2(\Omega))}
		\le
		C'\,\sqrt{T}\,\norm{\widetilde \rho_1 - \widetilde\rho_2}_{C([0,T];L^2(\Omega))}
		,
	\end{equation*}
	where now, due to the Gronwall argument in \cref{lem:ODE_reg}, $C'$ depends on $T$ but is monotonically decreasing. 
	Thus, we can again find $T$ small enough so that $S$ is a contraction and Banach's fixed point theorem asserts the existence of a unique solution.
	The box constraints $0 \le \rho_\gamma \le 1$ \ale in $Q_T$ follow by applying Lem.~3 in \cite{EggerPietschmannSchlottbom:2015:1}, \ie, by testing with smoothed versions of the positive part of $\rho-1$ and $-\rho$, respectively. In view of these uniform estimates a standard continuation argument yields existence for arbitrary $T>0$.

	So far we have shown
	\begin{equation*}
		\rho_\gamma \in L^2(0,T;H^1(\Omega))\cap H^1(0,T;H^1(\Omega)^*) \cap L^\infty(Q_T)
		,
	\end{equation*}
	since the fixed-point satisfies the weak formulation \eqref{eq:weak_form_rho_equation_reg}. 
	Note in particular that $\rho_\gamma$ is bounded in $L^2(0,T;H^1(\Omega))\cap H^1(0,T;H^1(\Omega)^*) \cap L^\infty(Q_T)$ by a constant independent of $\gamma$. 
	Furthermore, we denote the corresponding potential and agent trajectories by
	\begin{equation*}
		\phi_\gamma\in L^\infty(0,T;W^{2,p}(\Omega))\cap H^1(0,T;H^1(\Omega))
		,
		\quad 
		\bx_\gamma\in W^{1,\infty}(0,T)^M
		.
	\end{equation*}

	\underline{Step 2: Additional regularity:}
	To shorten the notation we write the nonlinear term in the form
	\begin{equation*}
		\rho_\gamma\,\beta(\rho_\gamma,\phi_\gamma,\bx_\gamma) 
		= 
		g(\rho_\gamma)\,h(\Phi_\gamma)
	\end{equation*}
	with $g(\rho_\gamma) \coloneqq \rho_\gamma\,f(\rho_\gamma)$ and $\Phi_\gamma \coloneqq \nabla(\phi_\gamma+\phi_K(\bx_\gamma;\cdot))$. 
	From the product rule we obtain the following representation for the divergence,
	\begin{equation*}
		\nabla\cdot\paren[auto](){\rho_\gamma\,\beta(\rho_\gamma,\phi_\gamma,\bx_\gamma)} 
		=
		g'(\rho_\gamma)\,\nabla\rho_\gamma \cdot h(\Phi_\gamma) + \rho_\gamma\,f(\rho_\gamma) \, \nabla \cdot h(\Phi_\gamma)
		.
	\end{equation*}
	Freezing the nonlinear terms allows us to understand \eqref{eq:rho_lin} as a linear equation of the form
	\begin{equation}
		\label{eq:rho_consider_linear_full}
		\begin{aligned}
			\partial_t \rho_\gamma - \eps\,\laplace \rho_\gamma + b(t,x)\cdot \nabla \rho_\gamma + c(t,x)\,\rho_\gamma
			&
			=
			0 
			&
			& 
			\text{in } Q_T 
			,
			\\
			\eps \nabla \rho_\gamma \cdot n 
			&
			= 
			r 
			&
			& 
			\text{on } \Sigma_T
			,
			\\
			\rho(0) 
			&
			= 
			\rho_0  
			&
			&
			\text{in } \Omega
		\end{aligned}
	\end{equation}
	with
	\begin{equation}
		\label{eq:full_def_b_c}
		\begin{aligned}
			b(t,x) 
			&
			\coloneqq 
			g'(\rho_\gamma)\,h(\Phi_\gamma) 
			\in 
			L^p(0,T;L^\infty(\Omega))
			,
			\\
			c(t,x) 
			&
			\coloneqq
			f(\rho_\gamma)\,\nabla\cdot h(\Phi_\gamma) 
			\in 
			L^p(Q_T)
			,
		\end{aligned}
	\end{equation}
	and
	\begin{equation}
		\label{eq:full_def_r}
		r(t,x) 
		\coloneqq
		- g(\rho_\gamma) \, h(\Phi_\gamma) \cdot n
		- \eta \, \rho_\gamma \, \chi_{\partial\Omega_\textup{D}} 
		\in 
		W^{\kappa,\kappa/2}_p(\Sigma_T)
	\end{equation}
	with $\kappa = 1-1/p$.
	It remains to show the regularity claimed for $b$, $c$ and $r$. 
	Again, this follows from the Hölder inequality and the regularity already shown for $\rho_\gamma$, $\phi_\gamma$, $\bx_\gamma$. 
	Together with the product and the chain rule this leads to
	\begin{align}
		\norm{b}_{L^\infty(Q_T)} 
		&
		\le 
		\norm{g'}_{L^\infty(\R)}\,\norm{h}_{L^\infty(\R^2)}\le C_b 
		,
		\label{eq:regularity_b}
		\\
		\norm{c}_{L^p(Q_T)} 
		&
		\le 
		\norm{f}_{W^{1,\infty}(\R)}\,\norm{Dh}_{W^{1,\infty}(\R^2)}
		\paren[auto](){\norm{\phi_\gamma}_{L^\infty(0,T;W^{2,p}(\Omega))} + C_{\phi_K}}
		\le 
		C_c
		.
		\label{eq:regularity_c}
	\end{align}
	Note that $C_b$ and $C_c$ are independent of $\gamma$ as, in particular, $\phi_\gamma$ can be bounded independently of $\rho_\gamma$ and thus of $\gamma$.
	To show the required regularity for $r$ we proceed as follows.
	First, we show the estimates
	\begin{align}\label{eq:bd_term_H1_bound}
		\MoveEqLeft
		\norm[auto]{\partial_{x_i}\paren[auto](){g(\rho_\gamma)\,h(\Phi_\gamma)}}_{L^2(Q_T)}
		\nonumber
		\\
		&
		\le 
		C \, \paren[Big](.{\norm{g}_{L^\infty(\R)}\,\norm{Dh}_{L^\infty(\R^{2\times 2})} \norm{\nabla^2(\phi_\gamma + \phi_K(\bx_\gamma;\cdot))}_{L^2(\Omega_T)}} 
		\nonumber
		\\
		&
		\qquad
		+ \paren[Big].){\norm{g'}_{L^\infty(\R)}\norm{\nabla \rho_\gamma}_{L^2(Q_T)}\norm{h}_{L^\infty(\R^2)}}
		\le 
		C
	\end{align}
	for $i=1,2$, as well as
	\begin{align}
		\label{eq:bd_term_temporal_bound}
		\MoveEqLeft
		\norm[auto]{\partial_t\paren[auto](){g(\rho_\gamma)\,h(\Phi_\gamma)}}_{L^2(Q_T)}
		\nonumber
		\\
		&
		\le 
		C \, \paren[Big](.{\norm{g}_{L^\infty(\R)}\norm{Dh}_{L^\infty(\R^{2\times 2})} \norm{\partial_t \nabla(\phi_\gamma + \phi_K(\bx_\gamma;\cdot))}_{L^2(Q_T)}}
		\nonumber
		\\
		&
		\qquad
		+ \paren[Big].){\norm{g'}_{L^\infty(\R)}\norm{\partial_t \rho_\gamma}_{L^2(Q_T)}\norm{h}_{L^\infty(\R^2)}}
		\le 
		C.
	\end{align}
	This follows from the regularity already shown for $\rho_\gamma$, $\phi_\gamma$ (see \cref{lem:ex_phi}) and $\bx_\gamma$ (see \cref{lem:ODE_reg}). 
	With these considerations we conclude
	\begin{equation*}
		-g(\rho_\gamma)\,h(\Phi_\gamma) - \eta\,\rho_\gamma\,\chi_{\partial\Omega_\textup{D}}\in W^{1,1}_2(Q_T) \embeds W^{1,1/2}_2(Q_T).
	\end{equation*}
	This allows us to apply the trace \cref{lem:trace_theorem}, which provides
	\begin{equation}
		\label{eq:regularity_r}
		\norm{r}_{W^{1/2,1/4}_2(\Sigma_T)} 
		\le 
		C_r
		.
	\end{equation}
	Collecting the properties \eqref{eq:regularity_b}--\eqref{eq:regularity_r} an application of Thm.~2.1 in \cite{DenkHieberPruess:2007:1} and \cref{lem:ex_phi} implies
	\begin{equation*}
		\rho_\gamma\in W^{2,1}_2(Q_T)
		\quad
		\text{and}
		\quad 
		\phi_\gamma\in L^\infty(0,T;H^3(\Omega))\cap H^1(0,T;H^1(\Omega))
		.
	\end{equation*}

	We can even further improve the regularity of $\rho_\gamma$. 
	Analogous to \eqref{eq:bd_term_H1_bound} and \eqref{eq:bd_term_temporal_bound} we show
	\begin{equation*}
		\norm[auto]{\partial_{x_i}\paren[auto](){g(\rho_\gamma)\,h(\Phi_\gamma)}}_{L^p(Q_T)}
		+ \norm[auto]{\partial_t\paren[auto](){\partial_{x_i}g(\rho_\gamma)\,h(\Phi_\gamma)}}_{L^2(0,T;L^p(\Omega)))}
		\le 
		C
	\end{equation*}
	and together with the trace \cref{lem:trace_theorem} we deduce
	\begin{align*}
		r
		&
		\in 
		L^p(0,T;W^{1,p}(\Omega))\cap W^{1,2}(0,T;L^p(\Omega)) \embeds W^{1,1/2}_p(Q_T) \embeds W^{1-1/p,1/2-1/(2p)}_p(\Sigma_T).
	\end{align*}
	This, \eqref{eq:regularity_b} and \eqref{eq:regularity_c} allow a further application of Thm.~2.1 in \cite{DenkHieberPruess:2007:1} with $p \le 4$, from which we infer the desired regularity
	\begin{equation*}
		\rho_\gamma\in W^{2,1}_p(Q_T).
	\end{equation*}

	The functions $\rho_\gamma$, $\phi_\gamma$ and $\bx_\gamma$ are thus a strong solution of the system
	\begin{alignat}{2}\label{eq:rho_reg}
		\partial_t\rho_\gamma - \nabla\cdot\paren[auto](){\rho_\gamma\,\beta(\rho_\gamma,\phi_\gamma,\bx_\gamma)}  
		&
		= 
		\eps\,\laplace \rho_\gamma 
		&
		&
		\quad
		\text{in } Q_T
		,
		\\
		- \delta_1\,\laplace \phi_\gamma + \abs{\nabla \phi_\gamma}^2 
		&
		= 
		\frac{1}{f( \rho_\gamma)^2+\delta_2}
		&
		&
		\quad
		\text{in } Q_T 
		,
		\label{eq:phi_reg}
		\\ 
		\label{eq:ODE_reg_reg}
		\dot x_{i,\gamma} 
		&
		= 
		f\paren[big](){(\overline{\eta_\gamma \ast\rho_\gamma})(\cdot,x_{i,\gamma}(\cdot))} \, u_i
		&
		&
		\quad
		\text{in } [0,T],
		\quad
		i = 1, \ldots, M
	\end{alignat}
	that satisfy the boundary and initial conditions pointwise almost everywhere.

	\underline{Step 3: Limit $\gamma \to 0$}: 
	In order to recover a solution to \eqref{eq:forward_system}, it remains to pass to the limit $\gamma \to 0$ in \eqref{eq:rho_reg}--\eqref{eq:ODE_reg_reg}. As a first step in this direction, note that $b(t,x)$, $c(t,x)$ and $r(t,x)$ defined in \eqref{eq:full_def_b_c} and \eqref{eq:full_def_r} are bounded independently of $\gamma$. 
	Thus, understanding \eqref{eq:rho_reg} as the linear equation \eqref{eq:rho_consider_linear_full} we obtain the following estimate, uniformly in $\gamma$,
	\begin{align*}
		\norm{\rho_\gamma }_{W^{2,1}_p(Q_T)} 
		\le 
		C
		.
	\end{align*}
	This implies the existence of a sequence $\gamma_k$ with $\gamma_k\to 0$ that
	\begin{equation}\label{eq:rho_gamma_limit}
		\rho_{\gamma_k} \weakly \rho \text{ in } W^{2,1}_p(Q_T)
		\quad
		\text{and} 
		\quad
		\rho_{\gamma_k} \to \rho \text{ in } L^p(0,T;W^{1,p}(\Omega)),
	\end{equation}
	where the second convergence is a consequence of the Aubin-Lions lemma, Thm.~II.5.16 in \cite{BoyerFabrie:2013:1}.
	As, moreover, $\phi_\gamma$ is also uniformly bounded in $L^p(0,T;W^{2,p}(\Omega))$ by $C_\phi$, we also have 
	\begin{equation}\label{eq:gamma_limit_phi}
		\begin{aligned}
			D^2 \phi_{\gamma_k}
			&
			\weakly D^2 \phi 
			&
			&
			\text{in } L^p(0,T;L^p(\Omega;\R^{d\times d}))
			,
			\\
			\phi_{\gamma_k} 
			&
			\to \phi 
			&
			&
			\text{in } L^p(0,T;W^{1,p}(\Omega))
			,
		\end{aligned}
	\end{equation}
	where the second convergence follows from \cref{lem:phi_gamma_compact}. We omit the index $k$ in the following to shorten the notation.
	Passing to the limit $\gamma \to 0$, in the sense of distributions,
	in equation \eqref{eq:forward_system2} yields the validity of this equation also for the limit values
	$\phi$ and $\rho$.
	As for \eqref{eq:forward_system3}, denote by $x_i(t)$ the solution of $\dot x_i(t) = f\paren[auto](){\overline\rho(x_i(t),t)}\,u_i(t)$, $t \in [0,T]$, and $x_i(0) = x_{i,0}$, with $\rho$ denoting the limit from \eqref{eq:rho_gamma_limit}. 
	Arguing similarly as in \cref{lem:ODE_reg} and using $\abs{u_i(t)} \le 1$ for \aa $t\in(0,T)$ we obtain
	\begin{align*}
		\MoveEqLeft
		\abs{x_i(t) - x_{i,\gamma}(t)} 
		\\
		&
		\le 
		\int_0^t \abs[auto]{f\paren[auto](){\overline\rho(s,x_i(s))} - f\paren[auto](){(\overline{\eta_\gamma \ast \rho_\gamma})(s,x_{i,\gamma}(s))}} \ds 
		\\
		&
		\le 
		\int_0^t \abs[auto]{f\paren[auto](){\overline\rho(s,x_i(s))} - f\paren[auto](){\overline\rho(s,x_{i,\gamma}(s))}} \ds 
		\\
		&
		\quad
		+ \int_0^t \abs[auto]{f\paren[auto](){\overline\rho (s,x_{i,\gamma}(s))} - f\paren[auto](){(\overline{\eta_\gamma \ast \rho})(s,x_{i,\gamma}(s))}} \ds
		\\
		&
		\quad
		+ \int_0^t \abs[auto]{f\paren[auto](){(\overline{\eta_\gamma \ast \rho})(s,x_{i,\gamma}(s))} - f\paren[auto](){(\overline{\eta_\gamma \ast \rho_\gamma})(s,x_{i,\gamma}(s))}} \ds
		\\
		&
		\le   
		\int_0^t L_{f,\rho,\E}(s)\,\abs{x_i(s)-x_{i,\gamma}(s)}\ds 
		\\
		&
		\quad
		+  L_f\paren[auto](){t\,\norm{\rho - \eta_\gamma \ast \rho}_{C([0,T]\times\closure{\Omega})} + \sqrt{t}\,\norm{\eta_\gamma}_{L^1(\Omega)}\,\norm{\rho - \rho_\gamma}_{L^2(0,T;L^\infty(\Omega))}}
		.
	\end{align*}
	As $\norm{\eta_\gamma}_{L^1(\Omega)}=1$, another application of Gronwall's inequality yields
	\begin{align*}
		\MoveEqLeft
		\abs{x_i(t) - x_{i,\gamma}(t)}
		\\
		&
		\le 
		C \, L_f \paren[auto](){t\,\norm{\rho - \eta_\gamma \ast \rho}_{C([0,T]\times\closure{\Omega})} + \sqrt{t}\,\norm{\rho - \rho_\gamma}_{L^2(0,T;L^\infty(\Omega))}} \exp\paren[auto](){\int_0^t L_{f,\rho,\E}(s) \ds},
	\end{align*}
	and due to \eqref{eq:convergence_mollification}, \eqref{eq:rho_gamma_limit} we see that as $\gamma \to 0$, the right-hand side converges to $0$ and thus 
	\begin{equation}\label{eq:x_gamma_limit}
		x_{i,\gamma}(t) \to x_i(t)\;\text{ for every } t \in (0,T), \quad i = 1, \ldots, M.
	\end{equation}
	Next, we pass to the limit in \eqref{eq:rho_reg}. 
	The convergence of the linear terms is a direct consequence of \eqref{eq:rho_gamma_limit}. 
	For the convection term we note that it can be written as
	\begin{align*}
		\nabla \cdot \paren[auto](){\rho_\gamma\,\beta (\rho_\gamma, \phi_\gamma, \bx_\gamma)}
		&
		= 
		\nabla \cdot \paren[auto](){g(\rho_\gamma)\, h(\Phi_\gamma)} 
		\\
		&
		= 
		g'(\rho_\gamma)\,\nabla\rho_\gamma \cdot h(\Phi_\gamma) + g(\rho_\gamma)\, Dh(\Phi_\gamma) \dprod \nabla^2(\phi_\gamma+ \phi_K(\bx_\gamma;\cdot))
		.
	\end{align*}
	As both $\rho_\gamma$ and $h$ are uniformly bounded, the convergences \eqref{eq:rho_gamma_limit}, \eqref{eq:gamma_limit_phi} and \eqref{eq:x_gamma_limit} imply
	\begin{equation*}
		\nabla \cdot (\rho_\gamma\,\beta (\rho_\gamma, \phi_\gamma, \bx_\gamma)) 
		\weakly 
		\nabla \cdot (\rho\,\beta (\rho, \phi, \bx))\quad\text{ in } L^p(Q_T)
		.
	\end{equation*}
	It remains to pass to the limit in the boundary conditions. 
	The trace \cref{lem:trace_theorem} implies that $\eps\,\nabla \rho_\gamma \cdot n$ is bounded in $L^p(\Sigma_T)$ and thus converges weakly. 
	The compact embedding $W^{2-1/p,1-1/(2p)}_p(\Sigma_T) \embeds L^p(\Sigma_T)$ implies the convergences 
	\begin{equation*}
		\restr{\rho}{\Sigma_T} \to \restr{\rho}{\Sigma_T}
		\quad
		\text{and}
		\quad
		\nabla\restr{\phi_\gamma}{\Sigma_T} \to \restr{\nabla\phi}{\Sigma_T}\;\text{ in } L^p(\Sigma_T)
		,
	\end{equation*}
	so that using the uniform boundedness of both $\rho_\gamma$ and $h$ as well as \eqref{eq:x_gamma_limit} we conclude 
	\begin{equation*}
		\eps\,\nabla\rho_\gamma\cdot n + \rho_\gamma\,\beta(\rho_\gamma,\phi_\gamma,\bx_\gamma)\cdot n
		\weakly 
		\eps\,\nabla \rho\cdot n + \rho\,\beta(\rho,\phi,\bx)\cdot n\;\text{ in } L^p(\Sigma_T)
		.
	\end{equation*}
	Thus, the weak limit of $\rho_\gamma$, namely $\rho$, is the strong solution of \eqref{eq:forward_system} with boundary conditions \eqref{eq:forward_bc}. 
	This completes the proof.
\end{proof}

The previous theorem allows us to introduce the solution operator
\begin{equation*}
	S \colon \cU \to \cY
	,
	\quad \bu \mapsto S(\bu) \coloneqq (\rho,\phi,\bx)
\end{equation*}
of \eqref{eq:forward_system} with boundary conditions \eqref{eq:forward_bc} and initial conditions $\bx(0) = \bx^0$ and $\rho(\cdot,0) = \rho_0$.
The control and state spaces are
\begin{align*}
	\cU 
	&
	\coloneqq
	L^\infty(0,T;\R^2)^M
	,
	\\
	\cY 
	&
	\coloneqq 
	W^{2,1}_p(Q)\times \paren[auto](){L^\infty(0,T;W_\text{ND}^{2,p}(\Omega))\cap W^{1,p}(0,T;W^{1,p}(\Omega))}\times W^{1,s}(0,T;\R^2)^M
\end{align*}
with $1/s=1/2+1/p$.
Later, the operator~$S$ is referred to as \emph{control-to-state operator}.

We conclude this section with an auxiliary result required later to show the existence of global solutions to an optimal control problem.
\begin{lemma}
	\label{lem:weak_continuity}
	The operator $S \colon \cU \to \cY$ is weakly sequentially continuous.
\end{lemma}
\begin{proof}
	Given a weakly convergent sequence $(u_n)_{n\in\N} \subset \cU$ with $u_n\weakly u$ one has to show that the corresponding states $(\rho_n,\phi_n,\bx_n) = S(u_n)$ converge weakly in $\cY$ to $(\rho,\phi,\bx) = S(u)$.
	This follows from the same arguments as those in step~3 of the proof of \cref{thm:ex_state}, together with the uniqueness of solutions.
\end{proof}

\section{The Linearized System}
\label{sec:linearized_system}

In order to prove necessary optimality conditions for the optimal control problems introduced later, we investigate the differentiability of the control-to-state operator~$S$.
The desired results follow from the implicit function theorem applied to the equation $e(\by,\bu) = 0$ with $\by = S(\bu) = (\rho,\phi,\bx)$. 
Here, $e$ corresponds to the strong formulation of our forward system \eqref{eq:forward_system1}--\eqref{eq:forward_system3}, more precisely, there holds
\begin{equation*}
	e \colon \cY \times \cU \to \cZ
\end{equation*}
with
\begin{multline*}
	\cZ 
	= 
	L^p(Q_T)\times W^{1-1/p,1/2-1/(2p)}_p(\Sigma_\textup{W}) \times W^{2(1-1/p),p}(\Omega) 
	\\
	\times \paren[auto](){L^\infty(0,T;L^p(\Omega))\cap W^{1,p}(0,T;W^{1,p}(\Omega)^*)}\times W^{1,s}(0,T;\R^2)^M
\end{multline*}
defined by
\begin{align}
	\label{eq:definition_functional_e}
	e_1(\by,\bu) 
	&
	\coloneqq
	\partial_t\rho - \varepsilon\,\laplace \rho - \nabla\cdot 
	(\rho\,\beta(\by)) 
	\nonumber
	\\
	e_2(\by,\bu) 
	&
	\coloneqq
	(\varepsilon \nabla \rho + \rho\,\beta(\by))\cdot n + \chi_{\partial\Omega_\textup{D}}\,\eta\,\rho
	\nonumber
	\\
	e_3(\by,\bu) 
	&
	\coloneqq
	\rho(0)-\rho_0 
	\\
	e_4(\by,\bu) 
	&
	\coloneqq
	-\delta_1\,\laplace \phi + \abs{\nabla \phi}^2 - \frac{1}{f(\rho)^2+\delta_2} 
	\nonumber
	\\
	e_{5,i}(\by,\bu)(t) 
	&
	\coloneqq 
	x_i(t) - x_{i,0} - \int_0^t f(\overline\rho(s,x_i(s)))\,u_i(s) \ds
	,
	\quad 
	i=1,\ldots,M
	.
	\nonumber
\end{align}
Notice that from here on we write $\beta(\by)$ in place of $\beta(\rho,\phi,\bx)$.
Recall that  $p\in(2,4)$ is a fixed number and the integrability index $s$ of the space of agent trajectories is chosen such that $\frac1s = \frac12 + \frac1p$.

\begin{lemma}
	\label{lem:important_bounds}
	Let $\by = (\rho,\phi,\bx) \in \cY$ and define $\Phi := \nabla(\phi+\phi_K(\bx;\cdot))$. 
	Then a constant $C_h>0$ exists such that
	\begin{align*}
		\norm{h(\Phi)}_{L^\infty(Q_T)} + \norm{h(\Phi)}_{L^\infty(0,T;W^{1,p}(\Omega))} +\norm{\partial_t h(\Phi)}_{L^p(Q_T)} 
		&
		\le 
		C_h 
		\\
		\norm{D^kh(\Phi)}_{L^\infty(Q_T)} + \norm{D^kh(\Phi)}_{L^\infty(0,T;W^{1,p}(\Omega))} 
		&
		\le 
		C_h,
		\quad 
		k=1,2
		.
	\end{align*}
	Note that
	\begin{equation*}
		Dh(\by) 
		= 
		\chi_{\{\abs{\by} \le 1\},\varepsilon} \, \id
	\end{equation*}
	is a smoothed characteristic function.	
	Moreover, the functions $h$ and $D^k h$, $k=1,2$, are Lipschitz continuous.
	In particular, for any two vector fields $\Phi_1, \Phi_2 \in L^\infty(0,T;W^{1,p}(\Omega))$ there holds
	\begin{align}
		\MoveEqLeft
		\norm{\nabla\cdot (Dh(\Phi_1)-Dh(\Phi_2))}_{L^\infty(0,T;L^p(Q_T))}
		\nonumber
		\\
		&
		\le 
		C_h
		\paren[auto](){1+\norm{D\Phi_1}_{L^\infty(Q_T)}}
		\norm{D\Phi_1-D\Phi_2}_{L^\infty(0,T;L^p(\Omega))}
		.
		\label{eq:divergence_Dh_Lipschitz}
	\end{align}
\end{lemma}
\begin{proof} 
	The result follows directly from the regularity properties of $h$.
\end{proof}

To shorten the notation we also introduce the following constants,
\begin{equation}
	\begin{aligned}
		C_\rho 
		&
		\coloneqq
		\norm{\rho}_{W^{2,1}_p(Q_T)}
		,
		&
		C_\phi 
		&
		\coloneqq 
		\norm{\phi}_{L^\infty(0,T;W^{2,p}(\Omega))} + \norm{\phi}_{W^{1,p}(0,T;W^{1,p}(\Omega))}
		,
		&
		\breakeqn{1}
		C_\bx 
		&
		\coloneqq
		\norm{\bx }_{W^{1,s}(0,T)}
		,
	\end{aligned}
\end{equation}
whose boundedness is guaranteed due to assumptions \cref{assumption:A4}, \cref{assumption:K1}, \cref{assumption:K2} and \cref{assumption:C1}. 

We confirm the assumptions of the implicit function theorem in the following lemmas.
\begin{lemma}
	\label{lem:e_differentiable}
	The operator $e \colon \cY \times \cU \to \cZ$ is continuously Fréchet differentiable.
\end{lemma}
\begin{proof} 	
	We start with equation $e_1$. 
	Again, we write the nonlinear term in the form
	\begin{equation*}
		\rho\,\beta(\by) = g(\rho)\,h(\Phi) 
		\text{ with } 
		\Phi = \nabla(\phi+\phi_K(\bx;\cdot))
		.
	\end{equation*}
	For given $(\rho,\phi, \bx) \in \cY$ and $(\widetilde \rho,\widetilde \phi, \widetilde \bx) \in \cY$, Taylor's formula with integral remainder yields
	\begin{align*}
		\MoveEqLeft
		e_1(\rho+\widetilde\rho,\phi,\bx;\bu) - e_1(\rho,\phi,\bx;\bu) 
		\\
		&
		= 
		\partial_t\widetilde\rho - \eps\laplace \widetilde \rho
		- \nabla\cdot (g'(\rho)\,\widetilde \rho\,h(\Phi))
		- \,\int_0^1 \nabla\cdot \paren[big][]{(g'(\rho+s\,\widetilde\rho) - 
		g'(\rho))\,\widetilde\rho\,h(\Phi)}\ds 
		,
		\\
		\MoveEqLeft
		e_1(\rho,\phi+\widetilde\phi,\bx;\bu) - e_1(\rho,\phi,\bx;\bu) 
		\\
		&
		=
		-\nabla\cdot \paren[big](){g(\rho)\,Dh(\Phi)\,\nabla\widetilde\phi} 
		- \int_0^1 \nabla \cdot \paren[big][]{g(\rho) 
		\paren[big](){Dh(\widetilde \Phi_\phi) - Dh(\Phi)} 
		\nabla\widetilde\phi} \ds 
		,
		\\
		\MoveEqLeft
		e_1(\rho,\phi,\bx+\widetilde\bx;\bu) - e_1(\rho,\phi,\bx;\bu) 
		\\
		&
		= 
		\nabla\cdot
		\paren[auto](){g(\rho)\,Dh(\Phi)\nabla^2K(\cdot-x_i)\,\widetilde x_i} 
		\\
		&
		\quad
		+ \int_0^1 \nabla\cdot\paren[big](){g(\rho) 
		\paren[big][]{Dh(\widetilde\Phi_\bx)\,\nabla^2K(\cdot-x_i-s\,\widetilde 
		x_i) - Dh(\Phi)\,\nabla^2K(\cdot-x_i)} \widetilde x_i} \ds
		,
	\end{align*}
	for $\widetilde\bx = (0,\ldots 0, \widetilde x_i, 0, \ldots 0)^\transp$, $i=1,\ldots,M$.

	In the above equations we used the notation $\widetilde\Phi_\phi = \nabla(\phi + s\,\widetilde\phi + \phi_K(\bx;\cdot))$ and $\widetilde \Phi_\bx = \nabla(\phi+\phi_K(\bx+s\,\widetilde \bx;\cdot))$.
	In the following we derive bounds for the remainder terms (the terms depending nonlinearly on $\widetilde\rho$, $\widetilde\phi$ and $\widetilde\bx$ in the above equations), which we denote by $r_{1,\rho}(\widetilde\rho)$, $r_{1,\phi}(\widetilde\phi)$ and $r_{1,x_i}(\widetilde\bx)$, $i=1,\ldots,M$.

	First, we apply the product rule, the Hölder inequality and the embedding $W^{2,1}_p(Q_T) \embeds L^\infty(Q_T)$, where we denote by $C_\infty$ the maximum of the embedding constant and $1$. 
	These arguments yield
	\begin{align}
		\label{eq:remainder_term_e1_rho}
		\MoveEqLeft
		\norm{r_{1,\rho}(\widetilde\rho)}_{L^p(Q_T)}
		\nonumber
		\\
		&
		\le 
		\int_0^1
		\Big(\paren[auto](){\norm{(g''(\rho+s\,\widetilde\rho) - g''(\rho))}_{L^\infty(Q_T)}\norm{\nabla\rho}_{L^p(Q_T)} + s\,\norm{g''(\rho+s\,\widetilde\rho)}_{L^\infty(Q_T)}\norm{\nabla\widetilde\rho}_{L^p(Q_T)}} 
		\nonumber
		\\
		& 
		\qquad 
		\times \norm{\widetilde\rho}_{L^\infty(Q_T)}\norm{h(\Phi)}_{L^\infty(Q_T)} \nonumber\\
		&
		\quad 
		+ \norm{g'(\rho+s\,\widetilde\rho)-g'(\rho)}_{L^\infty(Q_T)}\nonumber\\
		&
		\qquad
		\times \paren[auto](){\norm{\nabla\widetilde\rho}_{L^p(Q_T)}\,\norm{h(\Phi)}_{L^\infty(Q_T)} + \norm{\widetilde\rho}_{L^\infty(Q_T)} \norm{\nabla\cdot h(\Phi)}_{L^p(Q_T)}}\Big)\ds 
		\nonumber
		\\
		&
		\le 
		C_g\,C_h\paren[auto](){C_\rho\,+3\,C_\infty} \,\norm{\widetilde\rho}_{W^{2,1}_p(Q_T)}^2 
		= 
		\co \paren[big](){\norm{\widetilde\rho}_{W^{2,1}_p(Q_T)}}
		.
	\end{align}
	Second, to show differentiability with respect to $\phi$ we confirm
	\begin{align*}
		r_{1,\phi}(\widetilde\phi)
		&
		= 
		\int_0^1 \Big[g'(\rho)\nabla\rho \paren[big](){Dh(\widetilde\Phi_\phi) - Dh(\Phi)}\nabla\widetilde\phi 
		\\
		&
		\quad
		+ g(\rho) \paren[Big](){s\,\nabla \cdot \paren[big](){Dh(\widetilde\Phi_\phi) - Dh(\Phi)} \cdot \nabla\widetilde\phi + \paren[big](){Dh(\widetilde\Phi_\phi)-Dh(\Phi)} \dprod \nabla^2\widetilde\phi} \Big] \ds
		.
	\end{align*}
	With the Hölder inequality, the Lipschitz properties of $Dh$ and $D^2 h$, in particular \eqref{eq:divergence_Dh_Lipschitz}, and the usual embeddings we can show
	\begin{equation*}
		\norm{r_{1,\phi}(\widetilde\phi)}_{L^p(Q_T)} 
		\le
		C_g \, C_h \, C_\infty \, (C_\rho + 2) \, \norm{\widetilde\phi}^2_{L^\infty(0,T;W^{2,p}(\Omega))} 
		= 
		\co \paren[big](){\norm{\widetilde\phi}_{L^\infty(0,T;W^{2,p}(\Omega))}}
		.
	\end{equation*}
	Third, we derive an estimate for $r_{1,x_i}(\widetilde \bx)$ with $\widetilde \bx = (0,\ldots,0,\widetilde x_i,0,\ldots,0)$.
	We use the notation $\phi_K = \phi_K(\bx;\cdot)$ and $\widetilde\phi_K = \phi_K(\bx+s\widetilde\bx;\cdot)$ as well as $K_i = K(\cdot - x_i)$ and $\widetilde K_i = K(\cdot -x_i-s\,\widetilde x_i)$,
	reformulate the remainder term by applying the product and chain rule and obtain
	\begin{align*}
		\MoveEqLeft
		r_{1,x_i}(\widetilde x_i)
		\breakeqn{2}
		\le 
		\int_0^1 \Big(\abs{g'(\rho)\,\nabla\rho} \, 
		\paren[big](){\abs{Dh(\widetilde\Phi_\bx)}\,\abs{\nabla^2 K_i - \nabla^2 \widetilde K_i} + \abs{Dh(\widetilde\Phi_\bx)- Dh(\Phi)}\,\abs{\nabla^2 K_i}} 
		\\
		&
		\quad 
		+ g(\rho) \Big[\abs{\nabla\cdot Dh(\widetilde \Phi_\bx)}\,
		\abs{\nabla^2\widetilde K_i - \nabla^2 K_i}
		+ \abs{Dh(\widetilde\Phi)}\abs{\nabla^3\widetilde K_i - \nabla^3K_i} 
		\\
		&
		\quad 
		\phantom{g(\rho)\Big[}
		+ \abs{\nabla^3 K_i}\,\abs{Dh(\widetilde\Phi_\bx)-Dh(\Phi)}
		+ \abs{\nabla^2 K_i}\,\abs{\nabla\cdot
		(Dh(\widetilde\Phi_\bx)-Dh(\Phi))}\Big] \Big) \, \abs{\widetilde x_i} \ds
		.
	\end{align*}
	Here, $\abs{\cdot}$ is an arbitrary vector, matrix or tensor norm, depending on the argument. 
	With the Hölder inequality, the regularity of $(\rho,\phi,\bx)$, in particular $\rho\in W^{2,1}_p(Q_T) \embeds L^\infty(0,T;W^{1,p}(\Omega))\cap L^p(0,T;W^{1,\infty}(\Omega))$, and the Lipschitz properties of $h$ and $K$ we deduce
	\begin{equation*}
		\norm{r_{1,x_i}(\widetilde \bx)}_{L^p(Q_T)}
		\le 
		C_g\,(C_\rho+4)\,C_h\,C_{\phi_K}\,2\,(1+C_{\phi_K})\norm{\widetilde x_i}_{L^\infty(0,T;\R^2)}^2
		= 
		\co \paren[big](){\norm{\widetilde x_i}_{L^\infty(0,T;\R^2)}}
		.
	\end{equation*}

	The differentiability of $e_2$ can be shown with similar arguments. 
	The Taylor formula yields
	\begin{align*}
		\MoveEqLeft
		e_2(\rho+\widetilde\rho,\phi,\bx;\bu) - e_2(\rho,\phi,\bx;\bu)
		\\
		&
		= 
		(\varepsilon \nabla \widetilde\rho + \widetilde\rho\,\beta(\by))\cdot n + \chi_{\partial\Omega_\textup{D}}\,\eta\,\widetilde\rho 
		\\
		&
		\quad
		+ g'(\rho)\,\widetilde \rho\,h(\Phi)\cdot n + \int_0^1 (g'(\rho+s\,\widetilde\rho)-g'(\rho))\,\widetilde\rho\,h(\Phi)\cdot n\ds
		,
		\\
		\MoveEqLeft
		e_2(\rho,\phi+\widetilde\phi,\bx;\bu) - e_2(\rho,\phi,\bx;\bu)
		\\
		&
		= 
		g(\rho)\,Dh(\Phi)\,\nabla\widetilde\phi\,\cdot n
		+ \int_0^1 g(\rho)\paren[big](){Dh(\widetilde\Phi_\phi) - Dh(\Phi)}\nabla\widetilde\phi\cdot n \ds 
		,
		\\
		\MoveEqLeft
		e_2(\rho,\phi,\bx+\widetilde \bx;\bu) - e_2(\rho,\phi,\bx;\bu)
		\\
		&
		= 
		-g(\rho)\,Dh(\Phi)\,\nabla^2K(\cdot-x_i)\,\widetilde x_i\cdot n 
		\\
		&
		\quad
		-\int_0^1 g(\rho)\paren[auto](){Dh(\widetilde\Phi_{\bx})\,\nabla^2K(\cdot - x_i-s\,\widetilde x_i) - Dh(\Phi)\,\nabla^2 K(\cdot-x_i)}\,\widetilde x_i\cdot n \ds
		,
	\end{align*}
	and it remains to estimate the remainder terms, \ie, the last terms on the right-hand sides of the previous equations. 
	We abbreviate these terms by $r_{2,\rho}(\widetilde\rho)\cdot n$, $r_{2,\phi}(\widetilde\phi)\cdot n$ and $r_{2,\bx}(\widetilde\bx)\cdot n$, respectively.

	First, we show $\norm{r_{2,\rho}(\widetilde\rho)\cdot n}_{W^{1-1/p,1/2-1/(2p)}_p(\Sigma_T)} = \co \paren[big](){\norm{\widetilde\rho}_{W^{2,1}_p(Q_T)}}$. 
	We wish to apply the trace \cref{lem:trace_theorem}, for which we have to show the required time and space regularity of the extension onto $Q_T$. 
	First, note that there holds
	\begin{equation}\label{eq:r2rho_W1p_bound}
		\norm{r_{2,\rho}(\widetilde\rho)}_{L^p(0,T;W^{1,p}(\Omega))} 
		= 
		\co \paren[big](){\norm{\widetilde\rho}_{W^{2,1}_p(Q_T)}}
		,
	\end{equation}
	which can be concluded from the same arguments as in \eqref{eq:remainder_term_e1_rho}. 
	Moreover, for the time derivative we show
	\begin{align*}
		\partial_t r_{2,\rho}(\widetilde\rho) 
		&
		=
		\int_0^1 \Big( \paren[big][]{(g''(\rho+s\widetilde\rho) - g''(\rho))\,\partial_t\rho + s\,g''(\rho+s\widetilde\rho)\,\partial_t \widetilde\rho} \widetilde\rho\,h(\Phi) 
		\\
		&
		\qquad 
		+ \paren[big](){g'(\rho+s\,\widetilde\rho)- g'(\rho)} \, \paren[big](){\partial_t\widetilde \rho\,h(\Phi) + \widetilde\rho\,\partial_t h(\Phi)} \Big) \ds
	\end{align*}
	and with the usual arguments we obtain for $s^{-1} = 2^{-1}+p^{-1}$
	\begin{equation}
		\label{eq:r2rho_time_deriv_bound}
		\norm{r_{2,\rho}(\widetilde\rho)}_{W^{1,s}(0,T;L^p(\Omega))}
		\le 
		C_\infty\,C_g\,C_h\,(3+C_\rho)\,\norm{\widetilde\rho}_{W^{2,1}_p(Q_T)}^2
		=
		\co \paren[big](){\norm{\widetilde\rho}_{W^{2,1}_p(Q_T)}}
		.
	\end{equation}
	With the estimates \eqref{eq:r2rho_W1p_bound} and \eqref{eq:r2rho_time_deriv_bound}, the embedding
	\begin{equation}\label{eq:embedding_for_tracetheorem}
		L^p(0,T;W^{1,p}(\Omega))\cap W^{1,s}(0,T;L^p(\Omega)) \embeds W^{1,1/2}_p(Q_T)
	\end{equation}
	and the trace \cref{lem:trace_theorem} we deduce
	\begin{equation*}
		\norm{r_{2,\rho}(\widetilde\rho)\cdot n}_{W^{1-1/p,1/2-1/(2p)}_p(\Sigma_T)} 
		= 
		\co \paren[big](){\norm{\widetilde\rho}_{W^{2,1}_p(Q_T)}}
		,
	\end{equation*}
	which implies the differentiability of $e_2$ with respect to $\rho$.

	To show an estimate for $r_{2,\phi}(\widetilde\phi)\cdot n$ we proceed in a similar fashion. 
	With analogous arguments we deduce the estimates
	\begin{align*}
		\norm{\partial_{x_i} r_{2,\phi}(\widetilde \phi)}_{L^p(Q_T)}	
		&
		\le
		C\,\norm{\nabla\widetilde\phi}_{L^\infty(0,T;W^{1,p}(\Omega))}^2
		, 
		\\
		\norm{\partial_t r_{2,\phi}(\widetilde\phi)}_{L^s(0,T;L^p(Q_T))}
		&
		\le
		C\left(C_\rho\,\norm{\nabla\widetilde\phi}_{L^\infty(Q_T)}^2
		+ \norm{\partial_t \nabla\widetilde\phi}_{L^p(Q_T)}\,\norm{\nabla\widetilde\phi}_{L^\infty(Q_T)}\right)
	\end{align*}
	and exploit again \eqref{eq:embedding_for_tracetheorem} to arrive at
	\begin{equation*}
		\norm{r_{2,\phi}(\widetilde\phi)}_{W^{1-1/p,1/2-1/(2p)}_p(\Sigma_T)} 
		= 
		\co \paren[big](){\norm{\widetilde \phi}_{L^\infty(0,T;W^{2,p}(\Omega))} + \norm{\widetilde\phi}_{W^{1,p}(0,T;W^{1,p}(\Omega))}}
		,
	\end{equation*}
	which confirms the differentiability of $e_2$ \wrt $\phi$.

	An analogous procedure is used to deduce an estimate for the remainder term $r_{2,\bx}(\widetilde \bx)\cdot n$. 
	For each direction $\widetilde\bx = (0,\ldots 0, \widetilde x_i, 0, \ldots 0)^\transp$, $i=1,\ldots,M$, a direct calculation taking into account \cref{lem:important_bounds} and the Lipschitz continuity of the derivatives of $K$ and $h$ yields
	\begin{align*}
		\norm{\partial_{x_i} r_{2,\bx}(\widetilde \bx)}_{L^p(Q_T)}
		&
		\le 
		C\,\norm{\widetilde x_i}^2_{L^\infty(0,T;\R^2)} 
		,
		\\
		\norm{\partial_t r_{2,\bx}(\widetilde \bx)}_{L^s(0,T;L^p(\Omega))} 
		&
		\le 
		C\,\norm{\widetilde x_i}_{W^{1,s}(0,T;\R^2)}^2
		.
	\end{align*}
	Using again \eqref{eq:embedding_for_tracetheorem} and the trace \cref{lem:trace_theorem} leads to
	\begin{equation*}
		\norm{r_{2,x_i}(\widetilde \bx)\cdot n}_{W^{1-1/p,1/2-1/(2p)}_p(\Sigma_T)} = \co \paren[big](){\norm{\widetilde x_i}_{W^{1,s}(0,T;\R^2)}}
		,
	\end{equation*}
	which proves the differentiability of $e_2$ \wrt $\bx$.

	The component $e_3$ is trivially differentiable.
	For $e_4$ we show
	\begin{align*}
		\MoveEqLeft
		e_4(\rho+\widetilde\rho,\phi,\bx) - e_4(\rho,\phi,\bx) 
		\\
		&
		= 
		\frac{2 f(\rho)\,f'(\rho)}{(f^2(\rho)+\delta_2)^2} \widetilde \rho
		+ \int_0^1\paren[auto](){\frac{2\,f(\rho+s\widetilde\rho)\,f'(\rho+s\widetilde\rho)}{(f(\rho+s\widetilde\rho)^2+\delta_2)^2}-\frac{2\,f(\rho)\,f'(\rho)}{(f(\rho)^2+\delta_2)^2}}\widetilde\rho \ds, 
		\\
		\MoveEqLeft
		e_4(\rho,\phi+\widetilde\phi,\bx) - e_4(\rho,\phi,\bx) 
		\\
		&
		= 
		\delta_1\laplace \widetilde \phi + 2\nabla\phi^\transp \nabla\widetilde\phi + \abs{\nabla\widetilde \phi}^2 
		.
	\end{align*}
	Again, we denote the remainder terms (the terms which are nonlinear in $\widetilde\rho$ and $\widetilde\phi$) by $r_{4,\rho}(\widetilde\rho)$ and $r_{4,\phi}(\widetilde\phi)$.
	We show
	\begin{align*}
		\norm{r_{4,\phi}(\widetilde\phi)}_{L^\infty(0,T;L^p(\Omega))}
		&
		\le 
		C\norm{\nabla\widetilde\phi}_{L^\infty(Q_T)}^2
		\le 
		C\norm{\widetilde \phi}_{L^\infty(0,T;W^{2,p}(\Omega))}^2
		,
		\\
		\norm{\partial_t r_{4,\phi}(\widetilde\phi)}_{L^p(0,T;W^{1,p}(\Omega)^*)}
		&
		\le 
		C\norm{\nabla\widetilde\phi}_{L^\infty(0,T;W^{1,p}(\Omega))}\,\norm{\partial_t\nabla\widetilde \phi}_{L^p(Q_T)}
		.
	\end{align*}
	and both estimates together imply
	\begin{equation*}
		\norm{r_{4,\phi}(\widetilde\phi)}_{\cZ_4} 
		= 
		\co \paren[big](){\norm{\widetilde\phi}_{\cY_2}}
		.
	\end{equation*}
	With similar arguments we can show the differentiability of $e_4$ with respect to $\rho$.

	Finally we consider $e_5$, which is nonlinear in $\rho$ and $x_i$. 
	We confirm by a simple computation that
	\begin{align*}
		\MoveEqLeft
		e_5(\rho+\widetilde\rho,\phi,\bx;\bu;t) - e_5(\rho,\phi,\bx;\bu;t)
		\\
		&
		= 
		\int_0^t f'(\overline\rho(\tau,x_i(\tau)))\,\overline{\widetilde\rho}(\tau,x_i(\tau))\,u_i(\tau) \d\tau 
		\\
		&
		\quad
		+\int_0^1 \int_0^t \paren[auto](){f'\paren[big](){(\overline{\rho+s\,\widetilde\rho})(\tau,x_i(\tau))} - f'\paren[big](){\overline\rho(\tau,x_i(\tau))}} \overline{\widetilde\rho}(\tau,x_i(\tau))\,u_i(\tau) \d\tau \ds,
		\\
		\MoveEqLeft
		e_5(\rho,\phi,\bx+\widetilde\bx;\bu;t) - e_5(\rho,\phi,\bx;\bu;t) 
		\\
		&
		= 
		\widetilde x_i(t) - \int_0^t f'(\rho(\tau,x_i(\tau)))\,\nabla\rho(\tau,x_i(\tau)) \cdot \widetilde x_i(\tau)\,u_i(\tau) \d\tau
		\\
		&
		\quad
		- \int_0^1 \int_0^t \Big(f'(\overline\rho(\tau,(x_i+s\,\widetilde x_i)(\tau)))\,\nabla\overline\rho(\tau,(x_i+s\,\widetilde x_i)(\tau)) 
		\\
		&
		\qquad
		- f'(\overline\rho(\tau,x_i(\tau)))\,\nabla\overline\rho(\tau,x_i(\tau))\Big) \cdot \widetilde x_i(\tau) \, u_i(\tau) \d\tau \ds
	\end{align*}
	for $\widetilde \bx=(0 \ldots \widetilde x_i\ldots 0)^\transp$, $i=1,\ldots,M$.
	Again, the remainder terms (the terms depending nonlinearly on $\widetilde\rho$ and $\widetilde \bx$ in the above equation) are denoted by $r_{5,\rho}(\widetilde\rho)$ and $r_{5,\bx}(\widetilde\bx)$.

	Using the Lipschitz continuity of $f'$, assumption \cref{assumption:C1} and $\norm{\overline \rho}_{L^\infty(\R^2)} \le \norm{\rho}_{L^\infty(\Omega)}$ (see \cref{lem:extension}), we obtain
	\begin{equation}
		\label{eq:r5rho_estimate}
		\abs{r_{5,\rho}(\widetilde\rho;t)}	
		\le 
		C_f\,C_{\E,\infty}^2\,\norm{\widetilde\rho}_{L^2(0,T;L^\infty(\Omega))}^2
		= 
		\co \paren[big](){\norm{\widetilde\rho}_{W^{2,1}_p(Q_T)}}
		.
	\end{equation}
	With similar arguments and the Hölder inequality with $1/s>2/p$ we deduce for the temporal derivative of $r_{5,\rho}$
	\begin{align*}
		\MoveEqLeft
		\norm{\partial_t r_{5,\rho}(\widetilde\rho)}_{L^s(0,T;\R^2)}
		\\
		&
		\le
		\int_0^1 \norm[Big]{\paren[auto](){f'\paren[big](){(\overline{\rho+s\widetilde\rho})(\cdot,x_i(\cdot))}
		- f'(\overline\rho(\cdot,x_i(\cdot)))} \, \overline{\widetilde\rho}\paren[big](){\cdot,x_i(\cdot)} \, u_i(\cdot)}_{L^s(0,T;\R^2)} \ds 
		\\
		&
		\le 
		C_f\,\norm{\overline{\widetilde\rho}}_{L^p(0,T;L^\infty(\R^2))}^2
		\le 
		C_f\,C_{\E,\infty}^2\,\norm{\widetilde\rho}_{L^p(0,T;L^\infty(\Omega))}^2
		=
		\co \paren[big](){\norm{\widetilde\rho}_{W^{2,1}_p(Q_T)}}
		.
	\end{align*}

	For the remainder term $r_{5,\bx}(\widetilde \bx)$ with $\widetilde\bx = (0,\ldots 0, \widetilde x_i, 0, \ldots 0)^\transp$ we show
	\begin{align}
		\label{eq:r5x_initial_estimate}
		\MoveEqLeft
		\abs{r_{5,\bx}(\widetilde\bx;t)}
		\nonumber
		\\
		&
		\le 
		\norm{\widetilde x_i}_{L^\infty(0,T;\R^2)}
		\nonumber
		\\
		&
		\quad 
		\times \int_0^1\int_0^t \Big(\abs[auto]{f'\paren[big](){\overline\rho(\tau,(x_i+s\,\widetilde x_i)(t))} - f'(\overline\rho(\tau,x_i(t)))} \abs{\nabla\overline\rho(\tau, (x_i+s\,\widetilde x_i)(\tau))} 
		\nonumber
		\\
		&
		\quad \qquad 
		+ \abs{f'(\overline\rho(\tau,x_i(\tau)))}\abs[auto]{\nabla\overline\rho(\tau,(x_i+s\,\widetilde x_i)(\tau)) - \nabla\overline\rho(\tau,x_i(\tau))} \Big) \d\tau \ds
		.
	\end{align}
	Using Lipschitz estimates for $f'$ and the mean value theorem we conclude
	\begin{align}
		\label{eq:Lipschitz_f_prime_rho}
		\abs[auto]{f'(\overline\rho(\tau,(x_i+s\,\widetilde x_i)(\tau))) - f'(\overline\rho(\tau,x_i(\tau)))}
		&
		\le 
		C_f\,\text{Lip}(\overline\rho(\tau,\cdot))\,\abs{\widetilde x_i(\tau)}
		\nonumber
		\\
		&
		\le 
		C_f \, C_\E \, \norm{\rho(\tau,\cdot)}_{W^{2,p}(\Omega)}\,\abs{\widetilde x_i(\tau)},
		\\[0.5\baselineskip]
		\label{eq:Lipschitz_grad_rho}
		\abs[auto]{\nabla\overline\rho(\tau,(x_i+s\,\widetilde x_i)(\tau)) - \nabla\overline\rho(\tau,x_i(\tau))}
		&
		\le
		\norm{\nabla\overline\rho(\tau,\cdot)}_{C^{0,\alpha}(\R^2)}\,\abs{\widetilde x_i(\tau)}^\alpha 
		\nonumber 
		\\
		&
		\le 
		C_\E \, \norm{\rho(\tau,\cdot)}_{W^{2,p}(\Omega)}\,\abs{\widetilde x_i(\tau)}^\alpha.
	\end{align}
	In the above estimates we used $\norm{\overline\rho}_{C^{0,1}(\R^2)} \le \norm{\overline\rho}_{C^{1,\alpha}(\Omega)}$, the continuity of $\E\colon C^{1,\alpha}(\overline\Omega)\to C^{1,\alpha}(\R^2)$, see \cref{lem:extension}, and the embedding $W^{2,p}(\Omega) \embeds C^{1,\alpha}(\overline\Omega)$, which is valid for $\alpha\in (0,1/2]$ due to $p>2$.

	The insertion of \eqref{eq:Lipschitz_f_prime_rho} and \eqref{eq:Lipschitz_grad_rho} into \eqref{eq:r5x_initial_estimate} yields 
	\begin{align*}
		\abs{r_{5,\bx}(\widetilde\bx;t)}
		&
		\le
		C_\E \, C_f \, C_\rho \, \max\paren[big]\{\}{t^{1-2/p}, t^{1-1/p}} \paren[auto](){C_\rho\, \norm{\widetilde x_i}^2_{L^\infty(0,T;\R^2)} + \norm{\widetilde x_i}^{1+\alpha}_{L^\infty(0,T;\R^2)}} 
		\\
		&
		= 
		\co \paren[big](){\norm{\widetilde x_i}_{L^\infty(0,T;\R^2)}}
		.
	\end{align*}
	The above also uses an application of the Hölder inequality in time and $2/p < 1$.

	Finally, the Hölder inequality in time with $1/s = 1/2+1/p$ yields the estimate
	\begin{align*}
		\MoveEqLeft
		\norm{\partial_t r_{5,\bx}(\widetilde \bx)}_{L^s(0,T;\R^2)}
		\\
		&
		\le 
		C_f\Big(\norm{\overline\rho(\cdot,(x_i+s\,\widetilde x_i)(\cdot))
		- \overline\rho(\cdot,x_i(\cdot))}_{L^2(0,T)}\,
		\norm{\nabla\overline\rho(\cdot, (x_i+s\,\widetilde x_i)(\cdot))}_{L^p(0,T)} 
		\\
		&
		\qquad 
		+ \norm{\overline\rho(\tau,x_i(\tau))}_{L^2(0,T)} \norm{\nabla\overline\rho(\cdot,(x_i+s\,\widetilde x_i)(\cdot)) - \nabla\overline\rho(\cdot,x_i(\cdot))}_{L^p(0,T)}\Big) \norm{\widetilde x_i}_{L^\infty(0,T;\R^2)} 
		\\
		&
		= 
		\co \paren[big](){\norm{\widetilde x_i}_{L^\infty(0,T;\R^2)}}
		,
	\end{align*}
	where the last step follows again from \eqref{eq:Lipschitz_f_prime_rho} and \eqref{eq:Lipschitz_grad_rho}. 
	This confirms the partial differentiability of $e_5$.

	Collecting the previous estimates proves the partial Fréchet differentiability of the operator~$e$. 
	Based on the properties of $f$, $h$ and $\phi_K$ using again Lipschitz continuity and boundedness, the operators $\partial_\rho e$, $\partial_\phi e$ and $\partial_{x_i} e$ depend continuously on $(\rho,\phi,\bx)$, which implies the continuous Fréchet differentiability. 
\end{proof}

Next, we show that the operator $\partial_\by e(\by,\bu)$ is invertible.
\begin{lemma}
	The operator $\partial_\by e(\by;\bu)$ is bijective.
	That is, given $\by = (\rho, \phi, \bx)\in \cY$ and $F = (F_1,\ldots,F_5)^\transp\in \cZ$, the system $\partial_\by e(\by,\bu) \, \widetilde \by = F$ given by
	\begin{equation}
		\label{eq:e_y_system}
		\begin{aligned}
			\partial_t \widetilde \rho - \varepsilon\,\laplace \widetilde\rho 
			- \nabla 
			\cdot\paren[auto](){\widetilde\rho\,\beta(\by) + 
			\rho\,\frac{\partial\beta(\by)}{\partial \by}\widetilde \by} 
			&
			= 
			F_1 
			\quad
			\text{in } Q_T
			, 
			\\
			\paren[auto](){\varepsilon\,\nabla \widetilde\rho 
			+ \widetilde\rho\,\beta(\by) 
			+ \rho\,\frac{\partial\beta(\by)}{\partial 
			\by} \widetilde \by} \cdot n + 
			\chi_{\partial\Omega_\textup{D}}\,\eta\,\widetilde\rho
			&
			= 
			F_2 
			\quad
			\text{on } \Sigma_T
			,
			\\
			\widetilde \rho(0) 
			&
			= 
			F_3 
			\quad
			\text{in } \Omega
			,
			\\
			-\delta_1\,\laplace \widetilde\phi + 2 \nabla\phi^\transp \nabla\widetilde\phi + \frac{2 f(\rho)\,f'(\rho)}{(f^2(\rho)+\delta_2)^2} \widetilde \rho 
			&
			= 
			F_4 
			\quad
			\text{in } Q_T
			,
			\\
			\widetilde x_i(t)  - \int_0^t f'(\overline\rho(s,x_i(s)))
			\paren[auto](){\nabla\overline\rho(s,x_i(s))^\transp \widetilde x_i(s) + \overline{\widetilde\rho}(s,x_i(s)) }u_i(s) \ds 
			&
			= 
			F_{5,i}(t)
			,
		\end{aligned}
	\end{equation}
	for $t \in (0,T)$ and $i=1,\ldots,M$, possesses a unique solution 
	$\widetilde \by = (\widetilde\rho,\widetilde\phi,\widetilde\bx) \in \cY$. 
\end{lemma}
\begin{proof}
	The strategy of the proof is to apply Banach's fixed point theorem to the linear system \eqref{eq:e_y_system} to avoid technicalities that arise from the fact that $\widetilde \bx(t)$ depends on $\widetilde \rho$ non-locally in time. 
	To this end, we introduce three solution operators.

	First, there is $F_\rho\colon \cY \to W_p^{2,1}(Q_T)$ which maps $\widetilde \by = (\widetilde\rho,\widetilde\phi,\widetilde \bx)$ to $\widehat \rho\in W^{2,1}_p(Q_T)$, which is defined as the solution to 
	\begin{equation}
		\label{eq:linearized_rho_equation}
		\begin{aligned}
			\partial_t \widehat \rho- \varepsilon \laplace \widehat\rho - 
			\nabla \cdot\paren[auto](){\widetilde \rho\,\beta(\by) + 
			\rho\,\frac{\partial\beta(\by)}{\partial \by}\widetilde \by} 
			&
			= 
			F_1 
			&
			&
			\text{in } Q_T
			, 
			\\
			\paren[auto](){\varepsilon \nabla \widehat\rho + \widetilde 
			\rho\,\beta(\by) + \rho \,\frac{\partial\beta(\widetilde 
			\by)}{\partial \by}} \cdot n + 
			\chi_{\partial\Omega_\textup{D}}\,\eta\,\widetilde\rho
			&
			= 
			F_2 
			&
			&
			\text{on } \Sigma_\textup{W}
			,
			\\
			\widehat \rho(0) 
			&
			= 
			F_3 
			&
			&
			\text{in } \Omega
			.
		\end{aligned}
	\end{equation}
	Second, we define the operator $F_\phi\colon W^{2,1}_p(Q_T) \to L^\infty(0,T;W_\text{ND}^{2,p}(\Omega))\cap H^1(0,T;H^1(\Omega))$ which maps $\widetilde \rho$ to the solution $\widetilde\phi$ of
	\begin{equation}
		\label{eq:linearized_Eikonal}
		-\delta_1\laplace \widetilde\phi + 2 \nabla\phi^\transp\nabla\widetilde\phi + \frac{2 f(\rho)\,f'(\rho)}{(f^2(\rho)+\delta_2)^2} \widetilde \rho = F_4\quad\text{in}\ Q_T
		,
	\end{equation}
	together with the boundary conditions \eqref{eq:forward_bc} (these are incorporated in the function space).

	Third, we introduce $F_\bx\colon W^{2,1}_p(Q_T) \to W^{1,\infty}([0,T];\R^n)^M$ mapping $\widetilde \rho$ to the solution $\widetilde\bx=(\widetilde x_1,\ldots,\widetilde x_M)^\transp$ of
	\begin{equation*}
		\widetilde x_i(t) 
		= 
		\int_0^t f'\paren[big](){\overline\rho(s,x_i(s))} \paren[auto](){\nabla\overline \rho(s,x_i(s))^\transp \widetilde x_i(s) + \overline{\widetilde\rho}(s,x_i(s))} u_i(s) \ds + F_{5,i}(t)
		,
	\end{equation*}
	for $i=1,\ldots,M$.

	The idea of the proof is to apply Banach's fixed point theorem in the space $W^{2,1}_p(Q_T)$ to the operator
	\begin{align}\label{eq:def_fixed_linear}
		\widehat \rho = F(\widetilde\rho) \coloneqq
		F_\rho(\widetilde\rho,F_\phi(\widetilde\rho),F_{\bx}(\widetilde\rho)).
	\end{align} 
	As the operator $F$ is affine, it suffices to show the boundedness of the linear part of $F$ by a constant smaller than~$1$, which will imply that $F(\widetilde\rho)$ is a contraction. 

	\underline{Step 1: Estimate for $F_\phi$.}
	Due to $\phi(t)\in W^{2,p}(\Omega) \embeds W^{1,\infty}(\Omega)$ for $p>2$ and \aa $t\in (0,T)$, we confirm that $\nabla\phi\in L^\infty(Q_T)$ holds.
	Moreover, taking into account $f\in W^{1,\infty}(\R)$ and $\rho\in W^{2,p}(Q_T) \embeds L^\infty(Q_T)$, we get
	\begin{equation*}
		\norm[auto]{\frac{2 f(\rho)\,f'(\rho)}{(f^2(\rho)+\delta_2)^2}\,\widetilde \rho}_{L^\infty(0,T;L^p(\Omega))}
		\le 
		2\,C_f\,C_\rho\,\delta_2^{-2}\,\norm{\widetilde \rho}_{L^\infty(0,T;L^p(\Omega))}
		.
	\end{equation*}
	From Thm.~3.17 in \cite{Troianiello:1987:1}, see also Prop.~2 in \cite{KunischNeicPlankTrautmann:2019:1} for an application to the linearized Eikonal equation, we then deduce that the problem \eqref{eq:linearized_Eikonal} possesses for \aa $t\in (0,T)$ a unique solution $\widetilde\phi(t)\in W^{2,p}_\text{ND}(\Omega)$, which fulfills the inequality
	\begin{equation}
		\label{eq:continuity_F_phi}
		\norm{\widetilde\phi}_{L^\infty(0,T;W^{2,p}(\Omega))} 
		\le 
		C\paren[auto](){\norm{\widetilde\rho}_{L^\infty(0,T;L^p(\Omega))} + \norm{F_4}_{L^\infty(0,T;L^p(\Omega))}}
		.
	\end{equation}
	Furthermore, formal differentiation of \eqref{eq:linearized_Eikonal} with respect to time and exploiting $\partial_t F_4\in L^p(0,T;W^{1,p}(\Omega)^*)$ yields
	\begin{equation}
		\label{eq:continuity_F_phi_timederiv}
		\norm{\partial_t \widetilde \phi}_{L^p(0,T;W^{1,p}(\Omega))}
		\le 
		C\paren[auto](){\norm{\widetilde\rho}_{W^{2,1}_p(Q_T)} + \norm{F_4}_{L^\infty(0,T;L^p(\Omega))} + \norm{\partial_tF_4}_{L^p(0,T;W^{1,p}(\Omega)^*)}}
		.
	\end{equation}
	In the second estimate $C$ may depend on $T$, via norms of $\partial_t\phi$, $\partial_t \rho$, etc..
	Since this implies that $C$ is decreasing when $T$ is decreasing, this does not affect the contraction argument.

	\underline{Step 2: Estimate for $F_{\bx}$.}
	Next, we show the boundedness of the solution operator $F_{\bx}$. 
	To this end, we define $C_\rho(t)\coloneqq\norm{\rho(t,\cdot)}_{W^{2,p}(\Omega)}$ and deduce from the last equation in \eqref{eq:e_y_system}, using the Hölder inequality with $1/p+1/q=1$ and the continuity of $\E\colon W^{2,p}(\Omega) \embeds C^{1,\alpha}(\closure{\Omega})\to C^{1,\alpha}(\R^2)$, see \cref{lem:extension}:
	\begin{align*}
		\abs{\widetilde x_i(t)}
		&
		\le 
		\int_0^t\abs{f'(\overline\rho(s,x_i(s)))\left(\nabla\overline\rho(s,x_i(s))^\transp \widetilde x_i(s)
		+ \overline{\widetilde\rho}(s,x_i(s))\right)u_i(s)} \ds + \abs{F_{5,i}(t)}
		\\
		&
		\le 
		C_f\,C_\E\,C_\infty \paren[auto](){\int_0^t C_\rho(s)\,\abs{\widetilde x_i(s)}\ds + t^{1/q}\,\norm{\widetilde\rho}_{W^{2,1}_p(Q_t)}}
		+ \norm{F_{5,i}}_{L^\infty(0,t)}
		.
	\end{align*}
	Due to the Gronwall inequality we obtain
	\begin{equation*}
		\abs{\widetilde x_i(t)}
		\le
		\big(C_f \, C_\E \, C_\infty t^{1/q}\,\norm{\widetilde\rho}_{W^{2,1}_p(Q_T)} + \norm{F_{5,i}}_{L^\infty(0,t)}\big)\,
		\exp\paren[auto](){C_f\,C_\E\,C_\infty\int_0^t C_\rho(s)\ds}
		.
	\end{equation*}

	We can further estimate $\int_0^t C_\rho(s)\ds \le t^{1/q}\,\norm{\rho}_{W^{2,1}_p(Q_t)}$.
	The constants depending only on input data and the linearization point $\rho$ are shifted into the generic constants $c,C$, which depend on $T$ in monotonically decreasing way only. 
	This implies
	\begin{equation}
		\label{eq:continuity_rho_to_x_Linfty}
		\norm{\widetilde x_i}_{L^\infty(0,T)} \le C\,e^{c\,T^{1/q}}\paren[auto](){T^{1/q}\,\norm{\widetilde\rho}_{W^{2,1}_p(Q_T)} + \norm{F_{5,i}}_{L^\infty(0,T)}}
		.
	\end{equation}
	From now on we will choose $T>0$ sufficiently small such that $C\,T^{1/q}\,e^{cT^{1/q}} < 1$.

	Furthermore, we can show $L^s$-regularity for the time derivative of $\widetilde x_i$. 
	From the classical formulation of the ODE and the usual Hölder arguments we deduce
	\begin{align*}
		\abs{\dot{\widetilde x}_i(t)} 
		&
		= 
		\abs[auto]{f'(\overline\rho(t,x_i(t))) \paren[big](){\nabla\overline\rho(t,x_i(t))^\transp \widetilde x_i(t) + \overline{\widetilde\rho}(t,x_i(t))} u_i(t) + \dot{F_5}(t)} 
		\\
		&
		\le 
		C_f\,C_\E\paren[auto](){C_{\rho}(t)\,\abs{\widetilde x_i(t)} + C_{\widetilde\rho}(t)} + \abs{\dot{F}_{5,i}(t)}
	\end{align*}
	for a.a.\ $t\in [0,T]$.
	Taking the $L^s(0,T;\R^2)$-norm yields, after insertion of the estimate \eqref{eq:continuity_rho_to_x_Linfty} and an application of the Hölder inequality with $1/s = 1/2 + 1/p$,
	\begin{align}
		\label{eq:continuity_rho_to_x_W1infty}
		\norm{\widetilde x_i}_{W^{1,s}(0,T;\R^2)}
		&
		\le 
		C_f\,C_\E\,T^{1/2}\paren[auto](){\norm{\rho}_{W^{2,1}_p(Q_T)}\,\norm{\widetilde x_i}_{L^\infty(0,T;\R^2)} + \norm{\widetilde \rho}_{W^{2,1}_p(Q_T)}} 
		\breakeqn{3}
		+ \norm{F_{5,i}}_{W^{1,s}(0,T;\R^2)} 
		\nonumber 
		\\
		&
		\le 
		C\,T^{1/2}\,\norm{\widetilde\rho}_{W^{2,1}_p(Q_T)} + \paren[big](){1+C\,e^{c\,T^{1/q}}}\norm{F_{5,i}}_{W^{1,s}(0,T;\R^2)}
		.
	\end{align}

	\underline{Step 3: Estimate for $F_\rho$:}
	Next, we consider the solution operator of \eqref{eq:linearized_rho_equation}. 
	By \cref{lem:W2p_regularity} we have the a~priori estimate
	\begin{align}
		\label{eq:continuity_rho_equation_1}
		\norm{\widehat \rho}_{W^{2,1}_p(Q_T)}
		&
		\le 
		C\Bigg[\norm{F_1}_{L^p(Q_T)} + \norm{F_2}_{W^{1-1/p,1/2-1/(2p)}_p(\Sigma_\textup{W})} + \norm{F_3}_{W^{2(1-1/p),p}(\Omega)} 
		\nonumber
		\\
		& 
		+ \norm[auto]{\nabla\cdot\paren[auto](){\widetilde\rho\,\beta(\by) + \rho\frac{\partial\beta(\by)}{\partial y}\widetilde y}}_{L^p(Q_T)} 
		\nonumber
		\\
		&
		+ 
		\norm[auto]{\paren[auto](){\widetilde\rho\,\beta(\by)+\rho\,\frac{\partial\beta(\by)}{\partial
		 y}\widetilde y}\cdot 
		n+\chi_{\partial\Omega_\textup{D}}\,\eta\,\widetilde\rho}_{W^{1-1/p,1/2-1/(2p)}_p(\Sigma_T)}
		 \Bigg]
		.
	\end{align}
	It remains to discuss the last two terms on the right-hand side.
	First, we confirm with the product rule, the Hölder inequality and \cref{lem:important_bounds} that the estimate
	\begin{align}
		\label{eq:divergence_rho_beta}
		\norm{\nabla\cdot \paren[auto](){\widetilde \rho\,\beta(\by)}}_{L^p(Q_T)} 
		& 
		\le 
		C_h\,T^{1/p}\left(\norm{\nabla \widetilde\rho\,f(\rho) +
			\widetilde\rho\,f'(\rho)\,\nabla\rho}_{L^\infty(0,T;L^p(\Omega))} + 
		\norm{\widetilde \rho\, f(\rho)}_{L^\infty(Q_T)}\right)  
		\nonumber
		\\
		&
		\le 
		C\,T^{1/p}\,\norm{\widetilde \rho}_{W^{2,1}_p(Q_T)}
	\end{align}
	holds.
	Note that we also exploited regularity results for the linearization point, more precisely, $f(\rho),\rho\in W^{2,1}_p(Q_T) \embeds L^\infty(0,T;W^{1,p}(\Omega)) \embeds L^\infty(Q_T)$ (see \eqref{eq:w12q_embedding}), $\phi\in L^\infty(0,T;W^{2,p}(\Omega))$, and the assumptions on $f$, $h$ and $\phi_K$.

	Next, we show an estimate for the term involving $\rho\frac{\partial\beta(\by)}{\partial \by}\widetilde \by$ and we study all partial derivatives of $\beta$ separately. 
	For the derivative with respect to $\rho$ we get	
	\begin{align}
		\label{eq:divergence_rho_dbeta_drho}
		&
		\norm[auto]{\nabla\cdot\paren[auto](){\rho \frac{\partial\beta(\by)}{\partial\rho}\widetilde \rho}}_{L^p(Q_T)}
		= \norm{\nabla\cdot\paren[auto](){\rho\,f'(\rho)\,\widetilde\rho\,h(\Phi)}}_{L^p(Q_T)}
		\nonumber
		\\
		&
		\qquad
		\le  
		C_h\paren[auto](){\norm{\widetilde \rho\,\nabla \rho\,(f'(\rho) + \rho\,f''(\rho))
		+ \rho\,\nabla \widetilde \rho\,f'(\rho)}_{L^p(Q_T)} 
		+ \norm{\rho\,f'(\rho)\,\widetilde \rho}_{L^p(0,T;L^\infty(\Omega))}})  \nonumber\\
		&
		\qquad 
		\le 
		C\,T^{1/p} \norm{\widetilde\rho}_{W^{2,1}_p(Q_T)}
		,
	\end{align}
	where the last step follows from the Hölder inequality, $f\in W^{2,\infty}(\R)$, \cref{lem:important_bounds} and the embedding $W^{2,1}_p(Q_T) \embeds L^\infty(0,T;W^{1,p}(\Omega))\embeds L^\infty(Q_T)$, taking into account the regularity of the linearization point $(\rho,\phi,\bx)$.

	For the terms involving $\frac{\partial\beta}{\partial\phi}$ we first confirm by a simple computation
	\begin{equation*}
		\frac{\partial h(\Phi)}{\partial \phi}\widetilde \phi = Dh(\Phi) \nabla \widetilde \phi.
	\end{equation*}
	This implies with the product rule
	\begin{equation}\label{eq:estimate_rhs_deriv_phi_0}
		\norm[auto]{\nabla\cdot\paren[auto](){\rho\frac{\partial\beta(\by)}{\partial\phi}\widetilde\phi}}_{L^p(Q_T)}
		\le 
		\norm[auto]{g'(\rho)\,\nabla\rho\cdot \paren[auto](){Dh(\Phi)\,\nabla\widetilde\phi} + g(\rho)\,\nabla\cdot\paren[auto](){Dh(\Phi)\,\nabla\widetilde\phi}}_{L^p(Q_T)}
	\end{equation}
	with $g(\rho) \coloneqq \rho\,f(\rho)$ and $g'(\rho) = \rho\,f'(\rho)+f(\rho)$.
	Due to $\widetilde\phi\in L^\infty(0,T;W^{2,p}(\Omega))$ and thus, $\nabla\widetilde\phi\in L^\infty(0,T;W^{1,p}(\Omega)) \embeds L^\infty(Q_T)$, as well as $\nabla\rho\in L^\infty(0,T;L^p(\Omega))$ and \cref{lem:important_bounds}, we get for the first term on the right-hand side of \eqref{eq:estimate_rhs_deriv_phi_0}
	\begin{align*}
		\MoveEqLeft
		\norm[auto]{g'(\rho) \nabla\rho\cdot \paren[auto](){Dh(\Phi)\,\nabla\widetilde\phi}}_{L^p(Q_T)} 
		\\
		&
		\le 
		T^{1/p}\,C_g\,C_h\,\norm{\nabla \rho}_{L^\infty(0,T;L^p(\Omega))}
		\norm{\nabla \widetilde\phi}_{L^\infty(Q_T)}		
		\le 
		C\,T^{1/p}\,\norm{\widetilde\phi}_{L^\infty(0,T;W^{2,p}(\Omega))}
		.
	\end{align*}
	To bound the second term in \eqref{eq:estimate_rhs_deriv_phi_0}, we insert the identity
	\begin{equation*}
		\nabla\cdot \paren[auto](){Dh(\Phi)\nabla\widetilde\phi} 
		= 
		\paren[auto](){\nabla\cdot Dh(\Phi)}\cdot \nabla \widetilde\phi + Dh(\Phi):\nabla^2\widetilde\phi
	\end{equation*}
	and obtain, together with the bounds for $Dh$, from \cref{lem:important_bounds}
	\begin{align*}
		\norm{g(\rho)\,\nabla\cdot\paren[auto](){Dh(\Phi)\,\nabla\widetilde\phi}}_{L^p(Q_T)}
		&
		\le 
		C_g\,C_h\,T^{1/p}\,\paren[auto](){\norm{\nabla\widetilde\phi}_{L^\infty(Q_T)}
		+ \norm{\nabla^2 \widetilde\phi}_{L^\infty(0,T;L^p(\Omega))}}
		\\
		&
		\le 
		C\,T^{1/p}\,\norm{\widetilde\phi}_{L^\infty(0,T;W^{2,p}(\Omega))}
		.
	\end{align*}
	Insertion of the two previous estimates into \eqref{eq:estimate_rhs_deriv_phi_0} then yields
	\begin{equation}\label{eq:divergence_rho_dbeta_dphi}
		\norm[auto]{\nabla\cdot\paren[auto](){\rho\frac{\partial\beta(\by)}{\partial\phi}\widetilde\phi}}_{L^p(Q_T)} 
		\le 
		C\,T^{1/p}\,\norm{\widetilde\phi}_{L^\infty(0,T;W^{2,p}(\Omega))}
		.
	\end{equation}

	The derivative with respect to $x_i$, $i=1,\ldots,M$, is handled using the chain rule, 
	\cref{lem:important_bounds} and assumption \cref{assumption:K2}, leading to
	\begin{align}
		\label{eq:divergence_rho_dbeta_dx}
		\norm[auto]{\nabla\cdot \paren[auto](){\rho\frac{\partial\beta(\by)}{\partial x_i}\widetilde x_i}}_{L^p(Q_T)}
		&
		= 
		\norm{\nabla\cdot\paren[auto](){g(\rho)\,Dh(\Phi)\,\nabla^2K(\cdot-x_i)\,\widetilde x_i}}_{L^p(Q_T)} \nonumber\\
		&
		\le 
		C\,T^{1/p}\,\norm{\widetilde x_i}_{L^\infty(0,T;\R^2)}
		.
	\end{align}
	The estimates \eqref{eq:divergence_rho_beta}, \eqref{eq:divergence_rho_dbeta_drho}, \eqref{eq:divergence_rho_dbeta_dphi} and \eqref{eq:divergence_rho_dbeta_dx} give an estimate for the fourth term in \eqref{eq:continuity_rho_equation_1}, namely
	\begin{equation}\label{eq:divergence_terms_linearized_rho_eq}
		\norm[auto]{\nabla \cdot \paren[auto](){\widetilde\rho\,\beta(\by) + \rho\frac{\partial\beta(\by)}{\partial \by}\widetilde \by}}_{L^p(Q_T)}
		\le 
		C\,T^{1/p}\,\norm{\widetilde \by}_\cY
		.
	\end{equation}

	It remains to show the boundedness of the fifth term in \eqref{eq:continuity_rho_equation_1} and we apply a trace theorem as in the proof of \cref{lem:e_differentiable}, see \eqref{eq:embedding_for_tracetheorem}.  
	This requires us to show appropriate bounds in the $L^p(0,T;W^{1,p}(\Omega))$- and $W^{1,s}(0,T;L^p(\Omega))$-norms.

	First, when replacing $\nabla\cdot$ by an arbitrary partial derivative $\partial_{x_j}$ in the previous considerations, we get analogously to \eqref{eq:divergence_terms_linearized_rho_eq}
	\begin{equation}
		\label{eq:linearized_system_spatial_boundary_reg}
		\norm[auto]{\widetilde\rho\,\beta(\by) + \rho\,\frac{\partial\beta(\by)}{\partial \by} \widetilde \by}_{L^p(0,T;W^{1,p}(\Omega))} 
		\le 
		C\,T^{1/p}\,\norm{\widetilde \by}_\cY
		.
	\end{equation}
	Furthermore, we study the temporal derivative of the expressions occurring in the fifth term of \eqref{eq:continuity_rho_equation_1}.
	The Hölder inequality and \cref{lem:important_bounds} yield
	\begin{align*}
		\norm{\partial_t (\widetilde\rho\,\beta(\by))}_{L^s(0,T;L^p(\Omega))}
		&
		\le 
		T^{1/2}\,\norm{\partial_t \paren[auto](){\widetilde \rho\,f(\rho)\, h(\Phi)}}_{L^p(Q_T)} 
		\\
		&
		\le 
		C_f\,C_h\,T^{1/2} \paren[auto](){\norm{\partial_t\widetilde \rho}_{L^p(Q_T)}+\norm{\widetilde\rho}_{L^\infty(Q_T)}\,(\norm{\partial_t\rho}_{L^p(Q_T)} + 1)} 
		\\
		&
		\le 
		C_f\,C_h\,C_\infty\,(C_\rho+1)\,T^{1/2}\,\norm{\widetilde \rho}_{W^{2,1}_p(Q_T)}
		.
	\end{align*}
	For the term involving $\frac{\partial \beta(\by)}{\partial \rho}\widetilde \rho$ we get
	\begin{align*}
		&
		\norm[auto]{\partial_t (\rho\,\frac{\partial \beta(\by)}{\partial \rho}\widetilde\rho)}_{L^s(0,T;L^p(\Omega))}
		\le 
		T^{1/2}\,\norm{\partial_t (\rho\,f'(\rho)\,\widetilde \rho\,h(\Phi))}_{L^p(Q_T)}\\
		&
		\quad
		\le 
		C_f\,C_h\,T^{1/2}\,\paren[auto](){2\,\norm{\partial_t \rho}_{L^p(Q_T)}\,\norm{\widetilde\rho}_{L^\infty(Q_T)} + \norm{\rho}_{L^\infty(Q_T)}\,(\norm{\partial_t \widetilde\rho}_{L^p(Q_T)} +\norm{\widetilde\rho}_{L^\infty(Q_T)})}
		\\
		&
		\quad 
		\le 
		C_f\,C_h\,C_\rho\,C_\infty\,T^{1/2}\,\norm{\widetilde\rho}_{W^{2,1}_p(Q_T)}
		.
	\end{align*}
	For the term depending on $\frac{\partial \beta(\by)}{\partial \phi}\widetilde \phi$ we obtain
	\begin{align*}
		&
		\norm[auto]{\partial_t \paren[auto](){\rho\,\frac{\partial \beta(\by)}{\partial \phi}\widetilde\phi}}_{L^s(0,T;L^p(\Omega))}
		\le 
		T^{1/2}\,\norm{\partial_t(g(\rho)\,Dh(\Phi)\,\nabla\widetilde\phi)}_{L^p(Q_T)} 
		\\
		&
		\quad 
		\le 
		C_g\,C_h\,T^{1/2}\,\paren[auto](){\paren[big](){\norm{\partial_t \rho}_{L^p(Q_T)}+1}\norm{\nabla \widetilde\phi}_{L^\infty(Q_T)} + \norm{\partial_t\nabla \widetilde\phi}_{L^p(Q_T)}} 
		\\
		&
		\quad 
		\le 
		C_g\,C_h(2+C_\rho)\,T^{1/2}\,\paren[auto](){\norm{\widetilde\phi}_{L^\infty(0,T;W^{2,p}(\Omega))} + \norm{\partial_t \widetilde\phi}_{L^p(0,T;W^{1,p}(\Omega))}}
		.
	\end{align*}
	Finally, for the derivative with respect to $x_i$ we derive
	\begin{align*}
		\MoveEqLeft
		\norm[auto]{\partial_t \paren[auto](){\rho\,\frac{\partial \beta(\by)}{\partial x_i}\widetilde x_i}}_{L^s(0,T;L^p(\Omega))}
		\\
		&
		= 
		\norm{\partial_t(g(\rho)\,Dh(\Phi)\,\nabla^2K(\cdot-x_i)\,\widetilde x_i)}_{L^s(0,T;L^p(\Omega))} 
		\\
		&
		\le 
		C_g\,C_h\,C_K\,C_\bx\paren[auto](){(C_\rho+1)\,T^{1/2}\,\norm{\widetilde x_i}_{L^\infty(0,T;\R^2)} + 2\,\norm{\widetilde x_i}_{W^{1,s}(0,T;\R^2)}}
		.
	\end{align*}

	The previous four estimates and an embedding yield the desired property
	\begin{align}
		\label{eq:linearized_system_temporal_boundary_reg}
		\norm[auto]{\widetilde\rho\,\beta(\by)+\rho\,\frac{\partial\beta(\by)}{\partial \by}\widetilde \by}_{W^{1/2,p}(0,T;L^p(\Omega))} 
		&
		\le 
		C\,\norm[auto]{\widetilde\rho\,\beta(\by)+\rho\,\frac{\partial\beta(\by)}{\partial \by}\widetilde \by}_{W^{1,s}(0,T;L^p(\Omega))}
		\nonumber
		\\
		&
		\le 
		C\,T^{1/2}\paren[auto](){\norm{\widetilde \rho}_{\cY_1} + \norm{\widetilde \phi}_{\cY_2}} 
		+ C\,\norm{\widetilde \bx}_{\cY_3}
		.
	\end{align}

	Next, we may combine the estimates \eqref{eq:linearized_system_temporal_boundary_reg} and \eqref{eq:linearized_system_spatial_boundary_reg} and apply the trace \cref{lem:trace_theorem} to arrive at
	\begin{align}\label{eq:linearized_rho_eq_bd_terms}
		\MoveEqLeft
		\norm[auto]{\paren[auto](){\widetilde\rho\,\beta(\by)+\rho\,\frac{\partial\beta(\by)}{\partial
		 \by}\widetilde \by}\cdot 
		n+\chi_{\partial\Omega_\textup{D}}\,\eta\,\widetilde\rho}_{W^{1-1/p,1/2-1/(2p)}_p(\Sigma_T)}
		\nonumber 
		\\
		&
		\le
		C\,T^{1/p}\paren[auto](){\norm{\widetilde \rho}_{\cY_1} + \norm{\widetilde \phi}_{\cY_2}} 
		+ C\,\norm{\widetilde \bx}_{\cY_3}
		.
	\end{align}

	From the estimates \eqref{eq:divergence_terms_linearized_rho_eq} and \eqref{eq:linearized_rho_eq_bd_terms} we deduce that there exists a unique solution $\widehat \rho\in W^{2,1}_p(Q_T)$ of \eqref{eq:linearized_rho_equation}. 
	Moreover, together with \eqref{eq:continuity_rho_equation_1} using $p>2$, we obtain for sufficiently small $T>0$ the a~priori estimate
	\begin{align}
		\label{eq:continuity_F_rho}
		\norm{\widehat \rho}_{W^{2,1}_p(Q_T)}
		&
		\le 
		C\,\Big(T^{1/p}\paren[auto](){\norm{\widetilde \rho}_{\cY_1} + \norm{\widetilde \phi}_{\cY_2}} 
		+ \norm{\widetilde \bx}_{\cY_3}
		\nonumber
		\\
		&
		\qquad
		+ \norm{F_1}_{L^p(Q_T)} + \norm{F_2}_{W^{1-1/p,1/2-1/(2p)}_p(\Sigma_\textup{W})} + \norm{F_3}_{W^{2(1-1/p),p}(\Omega)}\Big)
		.
	\end{align}

	\underline{Step 4: The fixed-point argument}

	Now, we can combine the estimates \eqref{eq:continuity_F_phi}, \eqref{eq:continuity_F_phi_timederiv}, \eqref{eq:continuity_rho_to_x_W1infty} and \eqref{eq:continuity_F_rho} to deduce the continuity of the operator $F(\widetilde\rho)=F_\rho(\widetilde\rho,F_\phi(\widetilde\rho),F_\bx(\widetilde\rho))$. 
	This gives the estimate
	\begin{equation*}
		\norm{F(\widetilde\rho)}_{W^{2,1}_p(Q_T)} 
		\le 
		C\,\norm{F}_\cZ + C\,T^{1/p}\,\norm{\widetilde\rho}_{W^{2,1}_p(Q_T)}
	\end{equation*}
	for $T \le 1$.

	As $F$ is an affine linear operator we directly conclude the contraction property provided that $T$ is sufficiently small. 
	Thus, Banach's fixed-point theorem provides the existence of a unique solution $\widetilde \rho\in W^{2,1}_p(Q_T)$ which can be extended to an arbitrary time-scale with a concatenation argument.

	The regularities of the corresponding potential $\widetilde \phi$ and the agent trajectories $\widetilde x_i$, $i=1,\ldots,M$, follow from \eqref{eq:continuity_F_phi}, \eqref{eq:continuity_F_phi_timederiv} and \eqref{eq:continuity_rho_to_x_W1infty}.
\end{proof}

The previous two lemmas, together with the implicit function theorem, imply the differentiability of the control-to-state operator. 
The derivative in a direction $\widetilde\bu\in L^\infty(0,T;\R^2)^M$ can be computed by means of
\begin{equation*}
	e_\by(\by,\bu) \, \widetilde \by 
	= 
	- e_\bu(\by,\bu) \, \widetilde\bu
	.
\end{equation*}
We summarize the final result in the following theorem.

\begin{theorem}
	\label{thm:differentiability_S}
	The control-to-state operator $S \colon \cU \to \cY$ is Fréchet-differentiable and $\widetilde \by = S'(\bu) \, \widetilde \bu$ for given $\widetilde \bu \in \cU$ is characterized by the unique solution of 
	\begin{equation}
		\label{eq:linearized_system_strong}
		\begin{aligned}
			\partial_t \widetilde \rho - \varepsilon \laplace \widetilde\rho - 
			\nabla\cdot\left(\widetilde \rho \,\beta(\rho,\phi,\bx) + 
			\rho\left(\frac{\partial\beta(\rho,\phi,\bx)}{\partial\rho}\widetilde
			 \rho
					+ \frac{\partial\beta(\rho,\phi,\bx)}{\partial\phi}\widetilde \phi
			+ \frac{\partial\beta(\rho,\phi,\bx)}{\partial\bx}\widetilde \bx \right)\right)
			= 
			0
			,
			\\
			-\delta_1 \laplace \widetilde\phi + 2\nabla\phi\cdot\nabla \widetilde\phi + \frac{2 f(\rho)\,f'(\rho)}{(f^2(\rho)+\delta_2)^2} \widetilde \rho
			= 
			0 
			,
			\\
			\dot{\widetilde x}_i - f'(\overline\rho(\cdot,x_i))
			\left(\nabla\overline\rho(\cdot,x_i)^\transp \widetilde x_i 
			+ \overline{\widetilde\rho}(\cdot,x_i) \right)u_i 
			=
			f(\overline\rho(\cdot,x_i)) \, \widetilde u_i,
		\end{aligned}
	\end{equation}
	for $i=1,\ldots,M$, together with the boundary conditions \eqref{eq:forward_bc} and homogeneous initial conditions
	\begin{equation*}
		\widetilde\rho(\cdot,0) 
		= 
		0
		\quad
		\text{and}
		\quad
		\widetilde x_i(0) = 0,
		\quad
		i=1,\ldots,M.
	\end{equation*}
\end{theorem}

\section{Optimal Control Problem}\label{sec:opti}

In this section we come back to the optimal control problem outlined in \cref{subsec:optimal_control_problem}.
We consider objective functionals of the form
\begin{equation}
	\label{eq:general_objective}
	J(\rho,\phi,\bx;\bu) \coloneqq\Psi(\rho,\phi,\bx) + \frac{\alpha}{2\,T} \norm{\bu}_{H^1(0,T;\R^2)^M}^2,
\end{equation}
where $\alpha>0$ is a regularization parameter and the functional $\Psi\colon \cY\to \R$ fulfills the following assumption:
\begin{enumerate}[label=(J\arabic*)]
	\item\label{item:assumption_objective} 
		The functional $\Psi\colon\cY\to\R$ is weakly lower
		semi-continuous and bounded from below on
		$\{(\rho,\phi,\bx)\in \cY\colon \rho\ge 0\ \text{a.\,e.\ in}\ Q_T\}$.
\end{enumerate}
The assumed weak lower semi-continuity is fulfilled, \eg, when $\Psi$ is
convex and continuous. As the density part $\rho$ of each solution of the forward model
\eqref{eq:forward_system} is non-negative it suffices to assume boundedness of
$\Psi$ on a subset only.

We consider the following optimal control problem:
\begin{equation}
	\label{eq:optimal_control_problem}
	\left\lbrace
		\begin{aligned}
			\text{Minimize}
			\quad
			&
			J(\rho,\phi,\bx;\bu) \quad \text{where } (\rho,\phi,\bx; \bu) \in \cY\times\cU 
			\\
			\text{subject to}
			\quad
			&
			(\rho,\phi,\bx) = S(\bu)
			\\
			\text{and}
			\quad
			&
			\bu \in \Uad
			.
		\end{aligned}
	\right.
\end{equation}
The set of admissible controls is defined by
\begin{equation*}
	\Uad
	\coloneqq 
	\setDef[auto]{\bu\in \cU}{\norm{u_i}_{L^\infty(0,T)} \le 1 \text{ for } i=1,\ldots,M}
	.
\end{equation*}

This general framework covers in particular the two applications mentioned in 
\cref{sec:intro}

With standard arguments, see, \eg, Ch.~4.4 in \cite{Troeltzsch:2010:1}, \cref{item:assumption_objective} and \cref{lem:weak_continuity} imply the following existence result for \eqref{eq:optimal_control_problem}.

\begin{theorem}
	Let assumptions \cref{assumption:A1}--\cref{assumption:A4}, \cref{assumption:K1}--\cref{assumption:K2}, \cref{assumption:C1} and \cref{item:assumption_objective} hold. 
	Then problem \eqref{eq:optimal_control_problem} possesses at least one global solution $\bu^*\in \Uad$.
\end{theorem}

Moreover, with the properties shown for $S$ in \cref{sec:forward_system,sec:linearized_system}, in particular the Fréchet differentiability, we may deduce the following first-order necessary optimality condition, which is a direct consequence of the chain rule applied to the reduced optimization problem with objective $j(\bu) \coloneqq J(S(\bu);\bu)$ and subject to $\bu\in\Uad$:
\begin{theorem}
	\label{theorem:fonc}
	Let assumptions \cref{assumption:A1}--\cref{assumption:A4}, \cref{assumption:K1}--\cref{assumption:K2}, \cref{assumption:C1} and \cref{item:assumption_objective} hold. 
	Then each local minimizer $(\by^*,\bu^*)\in \cY\times \Uad$ of \eqref{eq:optimal_control_problem} fulfills
	\begin{equation*}
		\dual{\Psi'(\by^*)}{\delta \by}_{\cY^*\times \cY} + \frac{\alpha}{T} \, \inner{\bu^*}{\delta\bu}_{H^1(0,T)} 
		\ge 
		0
		\quad 
		\text{for all } \delta \bu \in \tangentcone{\Uad}{\bu^*},
	\end{equation*}
	where $\by^* = S(\bu^*)$ and $\delta \by = S'(\bu^*) \, \delta\bu$ and $S'$ is characterized by the system from \cref{thm:differentiability_S}.
	Moreover, $\tangentcone{\Uad}{\bu^*}$ denotes the tangent cone to $\Uad$ at $\bu^*$.
\end{theorem}

The optimality conditions from \cref{theorem:fonc} can be rewritten by expressing $\dual{\Psi'(\by^*)}{\delta \by}_{\cY^*\times \cY}$ in terms of a suitably defined adjoint state.
This then serves as the starting point for gradient-based optimization methods.
Such methods, as well as appropriate discretization of the forward and adjoint 
systems and numerical results will be discussed in a forthcoming publication.

\section*{Acknowledgments}
The authors would like to thank A.~St\"otzner (IAV GmbH) and M.-T.~Wolfram 
(Warwick) for fruitful discussions. 

\appendix
\section{Auxiliary Results}
\label{section:auxiliary_results}

First we present the proof of \cref{lem:H2_reg}. 
It is based on choosing derivatives of the solution as test functions and a Galerkin approximation as well as the uniqueness of weak solutions, see also Ch.~III, Section 6 in \cite{LadyzhenskayaUraltseva:1968:1}.
\begin{proof}[Proof of \cref{lem:H2_reg}]
	Existence in $L^2(0,T;H^1(\Omega)) \cap H^1(0,T;H^1(\Omega)^*)$ follows by a standard Galerkin approximation. 
	To this end we introduce $\rho^N(t,x) = \sum_{i=0}^N d_i(t)\varphi_i(x)$, where $(\varphi_i)_{i\in \N}$ is an orthonormal basis in $H^1(\Omega)$ which is also orthogonal \wrt $L^2(\Omega)$.
	The coefficients satisfy $d_i(0) = \inner{\rho_0}{\varphi_i}_{H^1(\Omega)}$. 
	Choosing $\rho^N$ as a test function results in 
	\begin{align}\label{eq:rho_eq_energy_estimate}
		\MoveEqLeft
		\int_0^T\frac{\d}{\dt} \int_\Omega (\rho^N)^2\dt\dx
		+ \eps\int_{Q_T}  \abs{\nabla \rho^N}^2\dt\dx
		\nonumber
		\\
		&
		\le 
		\int_{Q_T} \abs{\nabla \rho^N}\abs{g(\rho^N)\,h}\dt\dx -\eta \int_{\Sigma_\textup{D}}  (\rho^N)^2 \dt \ds_x 
		\nonumber
		\\
		&
		\le 
		\frac{\eps}{2} \int_{Q_T} \abs{\nabla \rho^N}^2\dt\dx
		+ \frac{1}{2\eps}\int_{Q_T} \abs{g(\rho^N) \, h}^2 \dt \dx,
	\end{align}
	\ie 
	\begin{align}\label{eq:rho_N_parabolic_reg}
		\norm{\rho^N}_{L^\infty(0,T;L^2(\Omega))} + \norm{\nabla \rho^N}_{L^2(Q_T)} 
		\le 
		C \paren[auto](){\norm{h}_{L^2(0,T;L^2(\Omega))} + \norm{\rho_0}_{L^2(\Omega)}},
	\end{align}
	where we used $\norm{\rho^N(\cdot,0)}_{L^2(\Omega)} \le C \, \norm{\rho_0}_{L^2(\Omega)}$. 
	This readily implies that $\rho^N$ is bounded, uniformly in $N$, in
	\begin{equation*}
		L^2(0,T;H^1(\Omega)) \cap H^1(0,T;H^1(\Omega)^*) \cap C([0,T];L^2(\Omega))
		,
	\end{equation*}
	see, \eg, Ch.~7.1.2, Thm.~2 in \cite{Evans:1998:1}. 
	A subsequent application of the Aubin-Lions Lemma (Thm.~II.5.16 in 
	\cite{BoyerFabrie:2013:1}), ensures the convergence of $\rho^N$ strongly in 
	$L^2(Q_T)$ to $\rho$ and thus the convergence of the nonlinear term 
	\begin{equation*}
		\int_{Q_T} g(\rho^N)\,h\cdot \nabla \xi \dt \dx
		\to 
		\int_{Q_T} g(\rho)\,h\cdot \nabla \xi \dt \dx
	\end{equation*}
	as $N\to \infty$ which follows by dominated convergence, noting that $g(\rho^N)\,h\cdot \nabla\xi$ is uniformly bounded in $L^1(0,T;L^1(\Omega))$. 
	Hence the limit satisfies the weak formulation \eqref{eq:rho_h_fixed}. 

	To show uniqueness, assume that there are two solutions $\rho_1$ and $\rho_2$. 
	Testing each equation with $\rho_1 - \rho_2$, integration w.r.t. to $t$, and using the Lipschitz continuity of $g$ together with an application of the weighted Cauchy inequality yields
	\begin{align*}
		\MoveEqLeft
		\frac{1}{2} \int_\Omega (\rho_1(t,\cdot)-\rho_2(t,\cdot))^2\dx
		+ 
		\frac{\eps}{2} \int_{Q_t} \abs{\nabla (\rho_1-\rho_2)}^2 \dt\dx
		\\
		&
		\le 
		\frac{1}{2\eps}\norm{h}_{L^\infty(0,t;L^\infty(\Omega))} \int_{Q_t} (\rho_1- \rho_2)^2\dt\dx
		- \eta \int_{\Sigma_{D,t}} (\rho_1-\rho_2)^2 \dt\ds_x
	\end{align*}
	for \aa $t \in (0,T)$. 
	As the last term on the right-hand side is non-positive, the uniqueness follows from Gronwall's inequality as both $\rho_1$ and $\rho_2$ have the same initial datum.

	To show the additional regularity, we choose, for finite $N$, the test function $\partial_t \rho^N$ and integrate \wrt time. 
	We obtain
	\begin{align}\label{eq:aux_regularity}
		\int_{Q_t} (\partial_t\rho^N)^2 \dt \dx
		&
		+ \eps\int_{Q_t} \nabla\rho^N\cdot \nabla (\partial_t\rho^N)\dt\dx
		- \int_{Q_t} g(\rho^N)\,h\cdot \nabla (\partial_t \rho^N)\dt\dx 
		\\
		&
		\qquad 
		= 
		- \eta \int_{\Sigma_{D,t}} \rho_N\,\partial_t\rho_N \dt\ds_x
		.
	\end{align}
	For the second term on the left-hand side we have
	\begin{align*}
		\eps\int_{Q_t} \nabla\rho^N\cdot \nabla (\partial_t\rho^N)\dt\dx
		&
		= 
		\frac{\eps}2\int_0^t\frac{\d}{\dt}\int_\Omega \abs{\nabla\rho^N}^2\dt\dx 
		\\
		&
		= 
		\frac{\eps}2\int_\Omega \abs{\nabla\rho^N(t,\cdot)}^2\dx -  \frac{\eps}2\int_\Omega \abs{\nabla\rho^N(0,x)}^2\dx
		,
	\end{align*}
	while for the third one, integration by parts in time together with the 
	(weighted) Cauchy's inequality gives
	\begin{align*}
		\MoveEqLeft
		-\int_{Q_t} g(\rho^N)\,h\cdot \nabla (\partial_t \rho^N)\dt \dx
		\\
		&
		= 
		\int_{Q_t} \partial_t \rho^N g'(\rho^N)\, h\cdot \nabla \rho^N\dt \dx
		+ \int_{Q_t} g(\rho^N)\,\partial_t h\cdot \nabla \rho^N\dt\dx
		- \int_{\Omega} \paren[Big].|{g(\rho^N)\,h \cdot \nabla \rho^N\dx}^t_0
		\\
		&
		\le 
		\frac{1}{2}\norm{\partial_t\rho^N}_{L^2(0,T;L^2(\Omega))}^2 + \frac{1}{2}\paren[auto](){\norm{g'(\rho^N)h}_{L^\infty(0,T;L^\infty(\Omega))}^2 + 1}\norm{\nabla \rho^N}_{L^2(0,T;L^2(\Omega))}^2 + \norm{g(\rho^N)\,\partial_t h}_{L^2(0,T;L^2(\Omega))}^2
		\\
		&\quad +\frac{\eps}{4}\norm{\nabla\rho^N(t,\cdot)}_{L^2(\Omega)}^2+ \frac{1}{4\eps}\norm{g(\rho^N)\,h}_{L^\infty(0,T;L^2(\Omega))}^2
		 + \norm{g(\rho(0,\cdot)^N)\,h(0,\cdot)\,\nabla 
		 \rho^N(0,\cdot)}_{L^2(\Omega)}^2.
	\end{align*}
	The term on the right-hand side of \eqref{eq:aux_regularity} becomes
	\begin{equation*}
		- \eta \int_{\Sigma_{D,t}} \rho_N \partial_t\rho_N \dt\ds_x 
		= 
		-\frac{\eta}{2} \int_\Omega (\rho^N(t,x))^2 \dx + \frac{\eta}{2} \int_\Omega (\rho^N(x,0))^2 \dx
		.
	\end{equation*}
	Combining all these estimates, an application of the trace theorem applied to $\norm{\rho_0}_{L^2(\partial\Omega_\textup{D})}$ and using 
	\begin{align*}
		\norm{\rho^N(\cdot,0)}_{H^1(\Omega)} 
		\le 
		\norm{\rho_0}_{H^1(\Omega)}
	\end{align*}
	yields
	\begin{equation}\label{eq:rho_higher_time_reg_before_gronwall}
		\frac{1}{2}\norm{\partial_t \rho^N}_{L^2(0,T;L^2(\Omega))}^2 + \frac{\eps}{4}\norm{\nabla \rho^N}_{L^\infty(0,T;L^2(\Omega))}^2
		\le 
		C\norm{\rho_0^N}_{H^1(\Omega)}^2 + c \, \norm{\nabla \rho^N}_{L^2(0,T;L^2(\Omega))}^2 + d
		,
	\end{equation}
	where
	\begin{align*}
		c 
		&
		= 
		\frac{1}{2}\paren[auto](){\norm{g'(\rho^N)\,h}_{L^\infty(0,T;L^\infty(\Omega))}^2 + 1} 
		,
		\\
		C
		&
		=
		\norm{g(\rho(0,\cdot)^N)\,h(0,\cdot)}_{L^\infty(\Omega)}^2,
		\\
		d 
		&
		= 
		\frac{1}{2}\norm{g(\rho^N)\,\partial_t h}_{L^2(0,T;L^2(\Omega))}^2 + 
		\frac{1}{4\eps}\norm{g(\rho^N)\,h}_{L^\infty(0,T;L^2(\Omega))}^2.
	\end{align*}
	Using \eqref{eq:rho_N_parabolic_reg}, we can estimate the second term on the right-hand side which yields the desired regularity for $\rho^N$. 
	We also infer the existence of a weakly converging subsequence in $L^2(0,T;L^2(\Omega))\cap L^\infty(0,T;H^1(\Omega))$ which, by the weak lower semi-continuity of the norms, yields the bound also for the limit~$\rho$.
	Reinserting this into \eqref{eq:rho_higher_time_reg_before_gronwall} yields the assertion.
\end{proof}
We also state the definition of Carathéodory conditions.
\begin{definition}\label{def:Caratheodory} We say that a function $r:Q_T \to \R^2$ satisfies the Carathéodory conditions whenever
	\begin{itemize}
		\item for each fixed $x \in \R^d$, the function $t \mapsto r(t,x)$ is measurable,
		\item for \aa $t \in [0,T]$, the function $x \mapsto r(t,x)$ is continuous,
		\item there exists an integrable function $m$ \st
			\begin{equation*}
				\abs{r(t,x)} 
				\le 
				m(t)\text{ for all } x \in \Omega \text{ and \aa\ } t \in [0,T]
				.
			\end{equation*}
	\end{itemize}
\end{definition}
\begin{lemma}\label{lem:conv_lipschitz} 
	Recall that $\eta_\eps$ is a standard mollifier.
	For given $\rho \in L^\infty(0,T;L^1(\Omega))$, the function $\rho_\eps \coloneqq \eta_\eps \ast \rho$ is Lipschitz continuous in $x$ with the Lipschitz constant depending on $\eps$ and $\norm{\rho}_{L^1(\Omega)} \in L^\infty(0,T)$, only.
\end{lemma}
\begin{proof}
	By definition, $\eta_\eps$ is Lipschitz continuous and we denote its Lipschitz constant by $L_\eps$. 
	Then, for \aa $t \in (0,T)$, we have
	\begin{align*}
		\abs{\rho_\eps(t,x) - \rho_\eps(t,y)} 
		&
		\le  
		\int_\Omega \abs{\eta_\eps(x-z) - \eta_\eps(y-z)}\abs{\rho(t,z)} \dz
		\\
		&
		\le 
		L_\eps \int_\Omega \abs{x-y}\abs{\rho(t,z)} \dz 
		\le 
		L_\eps \norm{\rho(t,\cdot)}_{L^1(\Omega)}\abs{x-y}
		.
	\end{align*}
\end{proof}

\begin{lemma}[Compactness of $\phi_\gamma$]\label{lem:phi_gamma_compact} 
	Fix $2 \le p < \infty$ and denote by $(\rho_{\gamma_k})_k \in 
	W^{2,1}_p(Q_T)$ a bounded family of functions, \ie there exists a constant 
	$C_\rho$, independent of $\gamma$, \st
	\begin{equation*}
		\norm{\rho_{\gamma_k}}_{W^{2,1}_p(Q_T)} 
		\le 
		C_\rho
		.
	\end{equation*}
	Denoting by $\phi_{\gamma_k}$ the sequence of solutions to 
	\eqref{eq:forward_system2} with $\rho_{\gamma_k}$ on the right-hand side, 
	there exists an element $\phi \in L^p(0,T;W^{1,p}(\Omega))$ and a 
	subsequence, again denoted by $\phi_{\gamma_k}$, \st
	\begin{align}\label{eq:phi_gamma_strong}
		\phi_{\gamma_k} \to \phi \text{ in } L^p(0,T;W^{1,p}(\Omega)).
	\end{align}
\end{lemma}
\begin{proof}
	Instead of \eqref{eq:forward_system2}, we work with the transformed version \eqref{eq:psi} and denote its solution by $\psi_\gamma$. 
	We seek to apply a variant of the Aubin-Lions Lemma, introduced in \cite{Simon:1978:1}, where instead of time derivatives we have to control finite differences in time. 
	Indeed, as $W^{2,p}(\Omega) \compactly W^{1,p}(\Omega)$, it suffices to show 
	\begin{equation*}
		\int_0^{T-h} \norm{\psi_\gamma(t+h,\cdot ) - \psi_\gamma(t,\cdot)}_{W^{1,p}(\Omega)}^p \dt 
		\le 
		\cO(h)
		.
	\end{equation*}
	To this end, note that the difference $\psi_\gamma(t+h,\cdot ) - \psi_\gamma(t,\cdot)$ satisfies the equation
	\begin{equation*}
		\laplace (\psi_\gamma(t+h,\cdot ) - \psi_\gamma(t,\cdot)) + q_t (\psi_\gamma(t+h,\cdot ) - \psi_\gamma(t,\cdot)) = (1-\psi(t,\cdot))(q_{t+h} - q_t),
	\end{equation*}
	with boundary conditions \eqref{eq:psi} and the corresponding a~priori estimate 
	\begin{equation*}
		\norm{\psi_\gamma(t+h,\cdot ) - \psi_\gamma(t,\cdot)}_{W^{1,p}(\Omega)} 
		\le 
		C(C_\phi)\norm{q_{t+h} - q_t}_{L^p(\Omega)}
		.
	\end{equation*}
	This yields
	\begin{equation}
	\label{eq:simon}
		\int_0^{T-h} \norm{\psi_\gamma(t+h,\cdot ) - \psi_\gamma(t,\cdot)}_{W^{1,p}(\Omega)}^p \dt 
		\le 
		C\int_0^{T-h} \norm{q_{t+h}(\cdot)- q_t(\cdot)}_{W^{1,p}(\Omega)}^p \dt
		.
	\end{equation}
	To estimate the right-hand side, note that $\rho_\gamma \in W^{1,p}(0,T;L^p(\Omega)) \embeds C^{0,1/2}([0,T];L^p(\Omega))$, and thus there exists a constant $L_\rho$, depending only on $C_\phi$ but not on $\gamma$, \st for all $t \ge 0$ and $h > 0$ satisfying $t+h \le T$, the inequality
	\begin{align*}
		\norm{\rho_\gamma(t+h,\cdot) - \rho_\gamma(t,\cdot)}_{L^p(\Omega)} 
		\le 
		L_\rho \sqrt{h}
	\end{align*}
	holds. 
	The definition of $q_t$ in \eqref{eq:def_qt} then directly implies
	\begin{align*}
		\norm{q_{t+h}(\cdot) - q_t(\cdot )}_{L^p(\Omega)} 
		\le 
		CL_\rho \sqrt{h}
		,
	\end{align*}
	which, when inserted into the right-hand side of \eqref{eq:simon}, yields 
	\begin{equation*}
		\int_0^{T-h} \norm{\psi_\gamma(t+h,\cdot ) - \psi_\gamma(t,\cdot)}_{W^{1,p}(\Omega)}^p \dt 
		\le 
		(CL_\rho)^p \, (T-h) \, h^{p/2}
		.
	\end{equation*}
	Since $2 \le p$ holds, this is the desired estimate. 
	Thus, Lemma on p.~1011 in \cite{Simon:1978:1} ensures the existence of a 
	subsequence that strongly converges in $L^p(0,T;W^{1,p}(\Omega))$.
	By inverting the transformation \eqref{eq:psitrans}, the result for $\phi_\gamma$ is obtained.
\end{proof}

%%% Local Variables:
%%% mode: latex
%%% TeX-master: "manuscript-cocv"
%%% End: